\documentclass[10pt]{article}
\usepackage{amssymb}
\usepackage[numbers]{natbib}
\usepackage{url}
\usepackage{graphicx}
\usepackage{graphics}
\usepackage{color}
\usepackage{xcolor}
\usepackage[figuresright]{rotating}
\usepackage{multirow}
\usepackage[margin=1in]{geometry}
\usepackage{algorithm}
\usepackage{algorithmicx}
\usepackage{algpseudocode}
\usepackage{amsmath}
\usepackage{dsfont}
\usepackage{amsthm}
\usepackage{amsfonts}
\usepackage{epic,eepic}
\usepackage{epsfig}
\usepackage{booktabs}
\usepackage[toc,page,title,titletoc,header]{appendix}
\usepackage[font={small}]{caption}
\usepackage{longtable}
\usepackage{pdflscape}
\usepackage[framed,numbered,autolinebreaks,useliterate]{mcode}
\usepackage{mathtools}
\usepackage{tabularx}
\usepackage{microtype}
\usepackage{bm}
\usepackage{hyperref}
\hypersetup{hidelinks}
\usepackage{tcolorbox}
\usepackage{enumerate}
\usepackage{wrapfig}
\usepackage{makecell}
\usepackage{subcaption}
\usepackage{multicol}
\usepackage{array}
\usepackage{algorithm}
\usepackage{algorithmicx}
\usepackage{algpseudocode}
\usepackage{pifont}
\usepackage{mdframed}

\allowdisplaybreaks

\def\bone{{\mathbf 1}}

\def\sR{{\mathbb R}}

\def\gS{{\mathcal{S}}}

\def\gO{{\mathcal{O}}}

\def\dist{{\rm dist}}

\makeatletter

\newcommand{\Rmnum}[1]{\expandafter\@slowromancap\romannumeral #1@}
\makeatother

\DeclareMathOperator*{\argmin}{arg\,min}

\theoremstyle{plain}
\newtheorem{theorem}{Theorem}[section]
\newtheorem{proposition}{Proposition}[section]
\newtheorem{lemma}{Lemma}[section]
\newtheorem{corollary}{Corollary}[section]
\newtheorem{definition}{Definition}[section]
\newtheorem{assumption}{Assumption}[section]
\newtheorem{example}{Example}[section]
\newtheorem{remark}{Remark}[section]

\title{An Adaptive Algorithm for Nonconvex Bilevel Optimization under the {Uniform} Polyak--{\L}ojasiewicz Condition}

\author{Xu Shi\textsuperscript{1}\quad Yinglin Du\textsuperscript{1}\quad Rufeng Xiao\textsuperscript{1}\quad Man-Chung Yue\textsuperscript{2}\quad Rujun Jiang\textsuperscript{$\dagger$, 1, 3}\\
\vspace{-2mm} \\
\normalsize{\textsuperscript{1}School of Data Science, Fudan University}\\
\normalsize{\textsuperscript{2}Department of Applied Mathematics, The Hong Kong Polytechnic University}\\
\normalsize{\textsuperscript{3}Shanghai Key Laboratory for Contemporary Applied Mathematics, Fudan University}\\
\vspace{-2mm} \\
\normalsize{\texttt{\{xshi22, yldu23, rfxiao24\}@m.fudan.edu.cn}}\\
\normalsize{\texttt{manchung.yue@polyu.edu.hk}}\\
\normalsize{\texttt{rjjiang@fudan.edu.cn}}\\
}
\date{}
\begin{document}
\maketitle
\begingroup
\begin{NoHyper}
\renewcommand\thefootnote{$\dagger$}
\footnotetext{Corresponding author.}
\end{NoHyper}
\endgroup

\begin{abstract}
Existing methods for nonconvex bilevel optimization (NBO) require prior knowledge of problem-specific parameters, such as Lipschitz moduli and the Polyak--{\L}ojasiewicz (P{\L}) constant, to configure stepsizes. This requirement severely limits the practical applicability as these parameters are often unknown or prohibitively expensive to estimate. To circumvent this limitation, we propose an adaptive first-order algorithm, named Adaptive Fully First-order Bilevel Approximation (AF${}^2$BA), for solving NBO problems under {the uniform P{\L} condition. The algorithm combines an outer backtracking mechanism based on inexact hyperobjective values with an adaptive auto-conditioned gradient method that converges linearly for the inner P{\L} problems.} We prove that {AF${}^2$BA generates an $\epsilon$-stationary point with a $\widetilde{\gO}(\ell\kappa^4\epsilon^{-2})$ first-order oracle complexity, which is near-optimal in its dependence on $\epsilon$: up to logarithmic factors, it matches the first-order lower bounds for nonconvex smooth optimization and nonconvex--strongly-convex smooth bilevel optimization. Moreover, the condition-number dependence of this complexity also matches that of the well-tuned F${}^2$BA algorithm under the uniform P{\L} condition, without requiring problem-specific constants in stepsize selection. Experiments demonstrate that AF${}^2$BA achieves comparable performance to existing well-tuned algorithms while exhibiting better robustness.}



\end{abstract}

\section{Introduction}
Bilevel optimization has attracted considerable attention due to its diverse applications in areas including reinforcement learning \citep{konda1999actor,hong2023two}, meta-learning \citep{bertinetto2018meta,rajeswaran2019meta}, hyperparameter optimization \citep{franceschi2018bilevel,shaban2019truncated,chen2024lower}, adversarial learning \citep{bishop2020optimal,wang2021fast,wang2022solving}, lexicographic optimization \citep{sabach2017first,gong2021bi,wang2024near,chen2024penalty}, and signal processing \citep{kunapuli2008classification,flamary2014learning}. The general bilevel problem is formulated as follows:
\begin{equation}
\label{p:primal}
\min\limits_{x\in\sR^{d_x}, y \in Y^*(x)} f(x,y) \quad {\mathrm{s.t.}} ~ Y^*(x)= \argmin\limits_{y\in\sR^{d_y}} g(x,y),
\end{equation}
where the functions $f$ and $g$ are called the upper- and lower-level objective functions, respectively. While extensive methodologies exist for regimes where the lower-level function $g$ is strongly convex \citep{ghadimi2018approximation,chen2021closing,ji2022will,hong2023two,kwon2023fully}, this requirement frequently proves overly restrictive for modern machine learning paradigms \citep{sinha2017review,hong2023two}. Consequently, this work investigates the generalized setting where $g$ lacks strong convexity.

To accommodate the potential non-uniqueness of $Y^*(x)$ in the absence of strong convexity, we consider the following hyperobjective reformulation of Problem \eqref{p:primal} \citep{dempe2002foundations,chen2024finding}:
\begin{equation}
\label{p:hyper}
\varphi(x) := \min_{y \in Y^*(x)} f(x,y).
\end{equation}
Since $\varphi(x)$ may not be convex, our goal is to find an $\epsilon$-stationary point \citep{ghadimi2018approximation,ji2021bilevel,ji2022will,kwon2023penalty,chen2024finding} of $\varphi(x)$ when it is differentiable. The definition is given as follows.
\begin{definition}
A point $x$ is said to be an $\epsilon$-stationary point of a differentiable function $\psi: \sR^d \rightarrow \sR$ if it satisfies $\| \nabla \psi(x) \| \le \epsilon$.
\end{definition}

When the lower-level function $g$ is strongly convex and twice continuously differentiable \citep{ghadimi2018approximation,ji2021bilevel,ji2022will}, the implicit function theorem \citep{dontchev2009implicit} ensures that $Y^*(x)$ reduces to a unique, differentiable singleton $y^*(x)$, and the hypergradient $\nabla \varphi(x)$ takes the analytical form:
\begin{equation}
\label{equ:hypersingle}
\nabla \varphi(x) = \nabla_x f(x,y^*(x)) - \nabla_{xy}^2 g(x,y^*(x)) \nabla_{yy}^2 g(x,y^*(x))^{-1} \nabla_y f(x,y^*(x)).
\end{equation}
This allows one to apply the hypergradient-based algorithms \citep{pedregosa2016hyperparameter,ghadimi2018approximation,franceschi2018bilevel,grazzi2020iteration,ji2021bilevel,ji2022will} to compute an $\epsilon$-stationary point for Problem \eqref{p:hyper}.

However, relaxing the strong convexity of $g$ introduces severe analytic challenges. Even when $g$ remains convex, the hyperobjective $\varphi(x)$ can become discontinuous and non-differentiable \citep{chen2024finding}, rendering the classical hypergradient formulation in \eqref{equ:hypersingle} ill-defined. {Furthermore, \cite[Theorem 3.2]{chen2024finding} constructs a class of bilevel problem instances whose lower-level functions are strictly convex and whose hypergradients are well-defined. Nevertheless, zero-respecting algorithms cannot find an $\epsilon$-stationary point within a prescribed finite iteration budget for these instances.}

To circumvent these issues, recent literature \citep{shen2023penalty,xiao2023generalized,kwon2023penalty,chen2024finding} leverages the Polyak--{\L}ojasiewicz (P{\L}) condition \citep{polyak1963gradient,lojasiewicz1963topological} (see Definition \ref{def:PL}). Specifically, when $g$ is P{\L} w.r.t. $y$, \cite{shen2023penalty} introduced the (value gap) penalty-based bilevel gradient descent (PBGD and V-PBGD) algorithms, which can attain an $\epsilon$-stationary point of the penalized objective $f + \gamma g$ (where $\gamma = \gO(\epsilon^{-0.5})$) with $\widetilde{\gO}(\epsilon^{-3})$ first-order oracle queries\footnote{Throughout this paper, a first-order oracle query consists of both a function value evaluation and a gradient evaluation.}. Incorporating second-order information, \cite{xiao2023generalized} proposed the generalized alternating method for bilevel optimization (GALET), securing an $\epsilon$-KKT point of Problem \eqref{p:primal} after at most $\widetilde{\gO}(\kappa^5\epsilon^{-2})$ first- and second-order oracle queries, where $\kappa > 0$ is the condition number (see Equation \eqref{equ:condinum}). With the supplementary assumption that the Hessian of $g$ is strictly positive-definite at any optimal solution of the lower-level problem, \cite{huang2023momentum} proposed the momentum-based gradient bilevel method (MGBiO), which can generate an $\epsilon$-stationary point of Problem \eqref{p:hyper} within $\gO(\kappa^6\epsilon^{-2})$ first- and second-order oracle queries. When the penalized objective $\sigma f + g$ is uniformly P{\L} w.r.t. $y$, \cite{kwon2023penalty} established the differentiability of $\varphi(x)$, and provided a proximal variant of F${}^2$BA \citep{chen2025near}, named Prox-F${}^2$BA, achieving $\widetilde{\gO}(\epsilon^{-3})$ first-order complexity. Under identical assumptions, \cite{chen2024finding} subsequently refined this bound, proving their F${}^2$BA algorithm computes an $\epsilon$-stationary point with a first-order oracle complexity of $\widetilde{\gO}(\kappa^4\epsilon^{-2})$, which is near-optimal in its dependence on $\epsilon$.

Beyond these approaches, specialized methods have also been developed for the case where $g$ is not strongly convex. Under specialized structural assumptions, \cite{liu2021value} and \cite{liu2021towards} established asymptotic convergence to the optimal value via the bilevel value-function-based interior-point method (BVFIM) and the initialization auxiliary pessimistic trajectory truncated gradient method (IAPTT-GM), respectively. Assuming the constant rank constraint qualification (CRCQ) and the P{\L} condition, \cite{liu2022bome} developed the bilevel optimization made easy (BOME) algorithm, yielding an $\epsilon$-KKT point of Problem \eqref{p:primal} in $\widetilde{\gO}(\epsilon^{-3})$ first-order oracle queries. Furthermore, for constrained or purely convex lower-level functions, \cite{lu2024first} proposed a penalty method, which converges to an approximate KKT point of Problem \eqref{p:primal}. For further discussions on these methods, please refer to \cite{kwon2023penalty,gao2022value,shen2023penalty,ye2023difference}.

The closest works to ours are the double-loop tuning-free bilevel optimizer (D-TFBO) and the single-loop tuning-free bilevel optimizer (S-TFBO) for Euclidean problems proposed by \cite{yang2025tuning}, and the adaptive Riemannian hypergradient descent (AdaRHD) method for Riemannian settings introduced by \cite{shi2025adaptive}, all designed for the case where the lower-level function is (geodesically) strongly convex. To compute an $\epsilon$-stationary point of Problem \eqref{p:hyper} or its Riemannian variant, D-TFBO and AdaRHD require $\gO(\epsilon^{-4})$ first-order oracle queries and $\widetilde{\gO}(\epsilon^{-4})$ second-order oracle queries, and S-TFBO achieves $\widetilde{\gO}(\epsilon^{-2})$ first- and second-order oracle queries. Overall, compared with D-TFBO, S-TFBO, and AdaRHD, our result establishes an adaptive, fully first-order guarantee under uniform P{\L} without requiring lower-level strong convexity.

Table \ref{table1} summarizes some key studies on bilevel optimization with non-asymptotic convergence rates, comparing their applicable scenarios, adaptivity, metrics, and the first- and second-order oracle complexities.

\begin{table*}[htp]
\centering
\caption{Comparisons of first- and second-order oracle complexities for different methods. Here, ``SC'' and ``P{\L}'' indicate that the lower-level function $g$ is strongly convex or satisfies the ordinary P{\L} condition, respectively; {``u-P{\L}'' denotes that $g+\sigma f$ satisfies the uniform P{\L} condition.} ``$\epsilon$-Stat'' and ``$\epsilon$-KKT'' correspond to the metrics $\epsilon$-stationary point and $\epsilon$-KKT point; ``Fir'' and ``Sec'' represent the requirement of first- and second-order oracles. The notation $\widetilde{\gO}$ suppresses logarithmic factors, while ``NA'' indicates that the second-order oracle is not required. We omit the condition-number dependence for methods whose original papers do not explicitly provide it.}
\label{table1}
\resizebox{0.8\textwidth}{!}{
\begin{tabular}{cccccc}
\hline
Methods & {Setting} & Adaptive & Metric & Fir & Sec \\
\hline
D-TFBO \citep{yang2025tuning} & \multirow{2}{*}{SC} & \multirow{2}{*}{$\boldsymbol{\checkmark}$} & \multirow{2}{*}{$\epsilon$-Stat} & $\gO(\epsilon^{-4})$ & $\gO(\epsilon^{-4})$\\
S-TFBO \citep{yang2025tuning} & & & & $\widetilde{\gO}(\epsilon^{-2})$ & $\widetilde{\gO}(\epsilon^{-2})$\\
\hline
AdaRHD-GD \citep{shi2025adaptive} & \multirow{2}{*}{SC} & \multirow{2}{*}{$\boldsymbol{\checkmark}$} & \multirow{2}{*}{$\epsilon$-Stat} & $\gO(\epsilon^{-4})$ & $\gO(\epsilon^{-4})$\\
AdaRHD-CG \citep{shi2025adaptive} & & & & $\gO(\epsilon^{-4})$ & $\widetilde{\gO}(\epsilon^{-2})$\\
\hline
BOME \citep{liu2022bome} & P{\L} & $\boldsymbol{\text{\ding{55}}}$ & $\epsilon$-KKT & $\widetilde{\gO}(\epsilon^{-3})$ & NA\\
\hline
PBGD (V-PBGD) \citep{shen2023penalty} & P{\L} & $\boldsymbol{\text{\ding{55}}}$ & $\epsilon$-Stat & $\widetilde{\gO}(\epsilon^{-3})$ & NA\\
\hline
GALET \citep{xiao2023generalized} & P{\L} & $\boldsymbol{\text{\ding{55}}}$ & $\epsilon$-KKT & ${\gO}(\kappa^{5}\epsilon^{-2})$ & $\widetilde{\gO}(\kappa^{5}\epsilon^{-2})$\\
\hline
MGBiO \citep{huang2023momentum} & P{\L} & $\boldsymbol{\text{\ding{55}}}$ & $\epsilon$-Stat & ${\gO}(\kappa^{6}\epsilon^{-2})$ & ${\gO}(\kappa^{6}\epsilon^{-2})$\\
\hline
Prox-F${}^2$BA \citep{kwon2023penalty} & {u-P{\L}} & $\boldsymbol{\text{\ding{55}}}$ & $\epsilon$-Stat & $\widetilde{\gO}(\epsilon^{-3})$ & NA\\
\hline
F${}^2$BA \citep{chen2024finding} & {u-P{\L}} & $\boldsymbol{\text{\ding{55}}}$ & $\epsilon$-Stat & $\widetilde{\gO}(\kappa^{4}\epsilon^{-2})$ & NA\\
\hline
{\textbf{{AF${}^2$BA} (Ours)}} & {u-P{\L}} & $\boldsymbol{\checkmark}$ & $\epsilon$-Stat & $\widetilde{\gO}(\kappa^{4}\epsilon^{-2})$ & NA\\
\hline
\end{tabular}
}
\end{table*}

\subsection{Contributions}
We propose the {Adaptive Fully First-order Bilevel Approximation (AF${}^2$BA)} algorithm for solving nonconvex bilevel optimization (NBO) problems under {the uniform Polyak--{\L}ojasiewicz (P{\L}) condition on $g+\sigma f$ over $\sigma\in[0,\bar{\sigma}]$, for some $\bar{\sigma}>0$. The algorithm uses outer backtracking based on inexact hyperobjective values and an adaptive gradient method for the inner P{\L} problems, with neither stepsize rule requiring prior knowledge of smoothness or P{\L} constants.} The core contributions of this work are summarized as follows:
\begin{enumerate}[(i)]
{
\item {For a sufficiently large prescribed horizon $T$}, AF${}^2$BA selects its inner and outer stepsizes without prior knowledge of problem-specific smoothness and P{\L} constants}, yet achieves an $\gO(\epsilon^{-2})$ outer iteration bound, matching the well-tuned F${}^2$BA algorithm \citep{chen2024finding}.

\item We formulate and analyze the auto-conditioned gradient method (AC-GM), a smooth variant of the auto-conditioned proximal gradient method (AC-PGM) \citep{yagishita2025simple}, for solving smooth P{\L} problems. {Building on the convergence analysis of AC-PGM, we establish an explicit linear convergence bound on the objective gap from the initial iterate, accounting for possible increases during unsuccessful iterations.}

\item We establish a first-order oracle complexity of {$\widetilde{\gO}(\ell\kappa^4\epsilon^{-2})$, which is near-optimal in its dependence on $\epsilon$ and matches the condition-number dependence of the well-tuned F${}^2$BA algorithm \citep{chen2024finding}, without using problem-specific constants in stepsize selection. The analysis combines explicit linear convergence bounds for the inner iterations with an outer backtracking argument that controls the errors in the hypergradients and hyperobjective values.}
\end{enumerate}

\section{Preliminaries}
This section formalizes the foundational definitions, assumptions, and preliminary theoretical results utilized throughout this paper.

\subsection{Definitions and assumptions}
Given a function $h: \sR^d \rightarrow \sR$, we denote $X_h^* := \argmin_{x \in \sR^d} h(x)$ and $h^* := \min_{x \in \sR^d} h(x)$. We first recall the Polyak--{\L}ojasiewicz (P{\L}) condition \citep{polyak1963gradient,lojasiewicz1963topological}.
\begin{definition}[P{\L} condition]
\label{def:PL}
A function $h: \sR^{d} \rightarrow \sR$ is said to be $\mu_h$-P{\L} for some constant $\mu_h > 0$ if for any $ x \in \sR^d$, it holds that
\begin{equation}
\label{equ:def:PL}
2 \mu_h (h(x) - h^*) \le \|\nabla h(x)\|^2.
\end{equation}
\end{definition}

As a substantial relaxation of strong convexity, the P{\L} condition accommodates non-convexity and permits the existence of multiple global minimizers \citep{chen2024finding}. {Specifically, examples of this ordinary P{\L} property arise in machine learning \citep{konda1999actor,sinha2017review,charles2018stability,li2018algorithmic,fazel2018global,liu2022loss,hong2023two,gaur2025sample}.}

To quantify the distance between solution sets, we utilize the Hausdorff distance \citep{chen2024finding,kwon2023penalty}.
\begin{definition}[Hausdorff distance]
Given two sets $S_1, S_2 \subseteq \sR^d$, the Hausdorff distance between $S_1$ and $S_2$ is defined as
\begin{equation*}
\dist_{H}(S_1,S_2) := \max\left\{\sup_{x_1\in S_1} \inf_{x_2\in S_2} \|x_1 - x_2\|, \sup_{x_2\in S_2} \inf_{x_1\in S_1}\|x_1 - x_2\|\right\}.
\end{equation*}
{Moreover, the distance between a point $x \in \sR^d$ and a set $S \subseteq \sR^d$ is defined as $\dist(x,S) := \inf_{s\in S} \|x - s\|$}.
\end{definition}

In the absence of strong convexity in the lower-level problem, the hypergradient $\nabla\varphi$ is not guaranteed to exist \citep{chen2024finding}. To rigorously address this potential non-differentiability, \cite{kwon2023penalty} demonstrated that the differential properties of $\varphi(x)$ can be analyzed via the following regularized hyperobjective:
\begin{equation}
\label{equ:varphisigma}
\varphi_{\sigma}(x) := \min_{y \in \sR^{d_y}} \left\{ f(x,y) + \frac{g(x,y) - g^*(x)}{\sigma}\right\},
\end{equation}
where $g^*(x) = \min_{y \in \sR^{d_y}} g(x,y)$.

Then, to ensure the differentiability of $\varphi_{\sigma}(x)$, \cite{kwon2023penalty} introduced a proximal error bound condition on the penalty function $g_{\sigma} := \sigma f + g$ for all $\sigma > 0$ near $0$, which was subsequently proved to be equivalent to the P{\L} condition for $g_{\sigma}$ \citep[Proposition D.1]{chen2024finding}. Our assumptions, consistent with those in \cite{kwon2023penalty,chen2024finding}, are formally stated below.
\begin{assumption}
\label{ass:basic}
The following statements hold.
\begin{enumerate}[(i)]
\item There exists a constant $\bar{\sigma} > 0$ such that for any $0 \le \sigma \le \bar{\sigma}$, the penalty function $g_{\sigma}(x,y) = \sigma f(x,y) + g(x,y)$ is uniformly $\mu$-P{\L} w.r.t. $y$, and its minimizer set w.r.t. $y$ is non-empty for every $x$;
\item The upper-level function $f(x,y)$ is $l_f$-Lipschitz and has $L_f$-Lipschitz gradients;
\item The lower-level function $g(x,y)$ has $L_g$-Lipschitz gradients;
\item The upper-level function $f(x,y)$ has $\rho_f$-Lipschitz Hessians w.r.t. $y$;
\item The lower-level function $g(x,y)$ has $\rho_g$-Lipschitz Hessians.
\end{enumerate}
\end{assumption}

{Whereas prior parameter-dependent algorithms \citep{kwon2023penalty,chen2024finding} only require Lipschitz continuity of $f$ w.r.t. $y$, our framework assumes this property w.r.t. $x$ as well.} This condition is necessary for bounding the hypergradient in our analysis (see Lemma \ref{lem:boundhyperg}) and represents a standard assumption in the literature on adaptive bilevel optimization \citep{yang2025tuning,shi2025adaptive}.
\begin{remark}
While Assumption \ref{ass:basic}(i) is stronger than assuming the P{\L} condition solely on $g$ {because adding $\sigma f$ need not preserve P{\L}}, it is a standard condition for achieving near-optimal first-order complexity in its dependence on $\epsilon$ \citep{chen2024finding,chen2025near}. Several fully first-order methods relying on the weaker assumption have complexity bounds such as $\widetilde{\gO}(\epsilon^{-3})$ \citep{liu2022bome,shen2023penalty}. Therefore, to establish a near-optimal complexity result for our adaptive framework, we adopt this uniform P{\L} requirement.
\end{remark}

\begin{assumption}
\label{ass:inf}
The hyperobjective $\varphi$ is bounded below, i.e.,
\[
\varphi^* := \inf_{x \in \sR^{d_x}} \varphi(x) > -\infty.
\]
\end{assumption}
Assumption \ref{ass:inf} concerns the existence of a lower bound for the hyperobjective $\varphi$, which is a standard condition in the convergence analyses of both adaptive single-level and bilevel algorithms \citep{ward2020adagrad,xie2020linear,chen2024finding,yang2025tuning,shi2025adaptive}.

{
We next introduce a synthetic bilevel problem satisfying Assumptions \ref{ass:basic} and \ref{ass:inf}.
\begin{example}\leavevmode\phantomsection\label{ex:synthetic}
Let $x\in\sR^2$, $y\in\sR^3$, $A=[I_2\ 0]$, and $b=(0.7,-0.4)^\top$. Consider the bilevel problem
\begin{equation*}
\begin{aligned}
\min_{\substack{x\in\sR^2}}\quad
& f(x,y)=b^\top(Ay-x)+q(x)+0.2\sum_{j=1}^{2}\bigl(1-\cos((Ay)_j)\bigr),\\
\text{s.t.}\quad
& y\in\argmin_{z\in\sR^3}g(x,z)=\tfrac12\|Az-x\|^2,
\end{aligned}
\end{equation*}
where
\[
q(x)=0.3\bigl(\sqrt{1+\|x\|^2}-1\bigr)+0.2\sum_{j=1}^{2}(1-\cos x_j).
\]
For $g_\sigma=g+\sigma f$, write $u=Ay$. Its Hessian w.r.t. $u$ satisfies
\[
\nabla^2_u g_\sigma
=I_2+0.2\sigma\operatorname{diag}(\cos u)
\succeq(1-0.2\sigma)I_2
\succeq0.8I_2,\qquad \sigma\in[0,1].
\]
Since $AA^\top=I_2$, $g$ is $1$-P{\L} and $g_\sigma$ is uniformly $0.8$-P{\L} w.r.t. $y$ for every $x$ and $\sigma\in[0,1]$. Thus, we can take $\bar\sigma=1$ and $\mu=0.8$.
The minimizer set of $g$ is nonempty, and the third coordinate of each minimizer is unrestricted:
\[
Y^*(x)=\{(x_1,x_2,c):c\in\sR\}.
\]
Hence, $g$ is not strongly convex w.r.t. $y$. Assumption \ref{ass:basic}(ii)--(v) holds with
\[
l_f=2,\qquad L_f=0.5,\qquad L_g=2,\qquad \rho_f=\rho_g=0.2.
\]
Here, $\rho_g$ is a valid upper bound since $\nabla^2g$ is constant. Finally,
\[
\varphi(x)=q(x)+0.2\sum_{j=1}^{2}(1-\cos x_j)\ge0,
\qquad \nabla\varphi(x)=\nabla q(x)+0.2\sin x,
\]
so Assumption \ref{ass:inf} also holds.
\end{example}
The learning configurations in Sections~\ref{sec:hyperclean} and~\ref{sec:hpo}, adapted from the data hyper-cleaning and feature-wise regularization tasks studied in \cite{shen2023penalty} and \cite{pedregosa2016hyperparameter}, respectively, satisfy the uniform P{\L} condition in Assumption~\ref{ass:basic}(i).
}

\subsection{Preliminary results}
We first recall some fundamental results under the P{\L} condition.
\begin{lemma}[{\citep[Theorem 2]{karimi2016linear}}]
\label{lem:PL}
If a function $h: \sR^d \rightarrow \sR$ is $\mu_h$-P{\L} and has $L_h$-Lipschitz gradients, then for any $x \in \sR^d$, it holds that
\begin{equation*}
\mu_h \dist(x,X_h^*) \le \|\nabla h(x)\| \le L_h \dist(x,X_h^*),
\end{equation*}
and
\begin{equation*}
\frac{\mu_h}{2} \dist^2(x,X_h^*) \le h(x) - h^*.
\end{equation*}
\end{lemma}

For any $0 \le \sigma \le \bar{\sigma}$, let $Y_{\sigma}^*(x) := \argmin_{y \in \sR^{d_y}} g_{\sigma}(x,y)$. The following lemma establishes the Lipschitz continuity of the solution set mapping $Y_{\sigma}^*(x)$ w.r.t. both $\sigma$ and $x$ in the Hausdorff distance.
\begin{lemma}[{\citep[Lemma 4.1]{chen2024finding}}]
\label{lem:Ylips}
Suppose that Assumption \ref{ass:basic} holds. Then, for any $0 \le \sigma_1, \sigma_2 \le \bar{\sigma}$, we have
\begin{equation*}
\dist_{H}( Y_{\sigma_1}^*(x_1), Y_{\sigma_2}^*(x_2)) \le \frac{l_f}{\mu} | \sigma_1 - \sigma_2 | + \frac{\sigma_1 L_f + L_g}{\mu} \| x_1 - x_2 \|.
\end{equation*}
\end{lemma}

As previously noted, the gradient of $\varphi(x)$ can be rigorously evaluated via the regularized function $\varphi_{\sigma}$. Specifically, \cite{kwon2023penalty} proved that $\nabla \varphi(x)$ can be obtained by taking the limit of $\nabla \varphi_{\sigma}$ as $\sigma \rightarrow 0$, while providing explicit bounds for the associated approximation errors.
\begin{lemma}[{\citep[Theorem 3.8]{kwon2023penalty}}]
\label{lem:nablalimit}
Suppose that Assumption \ref{ass:basic} holds. Then, $\nabla \varphi(x)$ exists and is given by
\begin{equation}
\label{equ:lem:nablalimit}
\nabla \varphi(x) = \lim_{\sigma \rightarrow 0^+} \nabla \varphi_{\sigma}(x) := \nabla_x f(x,y_{\sigma}^*(x)) + \frac{\nabla_x g(x,y_{\sigma}^*(x)) - \nabla_x g(x,y^*(x))}{\sigma},
\end{equation}
for any $y^*(x) \in Y^*(x)$, $y_{\sigma}^*(x) \in Y_{\sigma}^*(x)$.

Furthermore, for any $0 < \sigma \le \min\{\rho_g/\rho_f, \bar{\sigma}\}$, it holds that
\begin{equation*}
|\varphi_{\sigma}(x) - \varphi(x)| \le D_{\sigma} \sigma, \quad \text{and} \quad
\|\nabla \varphi_{\sigma}(x) - \nabla \varphi(x)\| \le \bar{D}_{\sigma} \sigma,
\end{equation*}
where $D_{\sigma}, \bar{D}_{\sigma} > 0$ are constants independent of $\sigma$, as defined in \cite{kwon2023penalty}.
\end{lemma}

Beyond the limiting expression, an explicit, closed-form expression for the hypergradient $\nabla \varphi(x)$ and its Lipschitz continuity are provided by \cite{chen2024finding}.
\begin{lemma}[{\citep[Lemma 4.4]{chen2024finding}}]
\label{lem:nablavarphi}
Suppose that Assumption \ref{ass:basic} holds. Then, the gradient $\nabla \varphi(x)$ has the following form
\begin{equation}
\label{equ:lem:nablavarphi}
\nabla \varphi(x) = \nabla_x f(x,y^*(x)) - \nabla_{xy}^2 g(x,y^*(x)) \left(\nabla_{yy}^2 g(x,y^*(x)) \right)^\dagger \nabla_y f(x,y^*(x)),
\end{equation}
for any $y^*(x) \in Y^*(x)$. Here $(\cdot)^\dagger$ represents the Moore--Penrose inverse \citep{penrose1955generalized}.

Furthermore, $\nabla\varphi(x)$ is $L_{\varphi}$-Lipschitz continuous, where
\[
L_{\varphi} := \left( L_f + \frac{l_f \rho_g}{\mu} \right) \left( 1 + \frac{L_g}{\mu} \right)^2.
\]
\end{lemma}

\section{Adaptive Algorithm for Nonconvex Bilevel Optimization}
\label{sec:Ada}
In this section, we introduce the {Adaptive Fully First-order Bilevel Approximation (AF${}^2$BA) algorithm for solving nonconvex bilevel optimization (NBO) problems under the uniform P{\L} condition on $g+\sigma f$. The outer stepsize is selected by backtracking based on
approximate hyperobjective values, while \hyperref[alg:AC-GM]{AC-GM} approximately solves the lower-level and penalty problems.} The pseudocode is presented in Algorithm \ref{alg:AF2BA}.

{
\begin{algorithm}[htp!]

\caption{{\textbf{A}daptive \textbf{F}ully \textbf{F}irst-order \textbf{B}ilevel \textbf{A}pproximation (\textbf{AF${}^2$BA})}}
\label{alg:AF2BA}
\begin{algorithmic}[1]
\State \textbf{Input:} Initial points $x_0,y_0,z_0$ and horizon $T\ge1$.
\State Set $\sigma=T^{-1}$, $\epsilon_y=\epsilon_z=T^{-2}$, $a=1$, $\alpha=2$, and $L_{0,1}=L_{0,2}=1$.\label{alg:AF2BA:sigma}
\State $(z_0^{K_0},K_0)=\hyperref[alg:AC-GM]{\textbf{AC-GM}}(g(x_0,\cdot),z_0,\alpha,L_{0,1},\epsilon_z)$.
\State Set $V_0=f(x_0,z_0^{K_0})$.
\For{$t=0,\ldots,T-1$}
\State Set $y_t^0=y_{t-1}^{N_{t-1}}$ if $t>0$ and $y_0$ otherwise.
\State $(y_t^{N_t},N_t)=\hyperref[alg:AC-GM]{\textbf{AC-GM}}(g_\sigma(x_t,\cdot),y_t^0,\alpha,L_{0,2},\epsilon_y)$.
\State $\widehat{\nabla}\varphi(x_t,y_t^{N_t},z_t^{K_t})=\nabla_x f(x_t,y_t^{N_t})+(\nabla_x g(x_t,y_t^{N_t})-\nabla_x g(x_t,z_t^{K_t}))/\sigma$.\label{alg:AF2BA:widehat}
\State $x^+=x_t-\widehat{\nabla}\varphi(x_t,y_t^{N_t},z_t^{K_t})/a$.
\State $(z_{t+1}^{K_{t+1}},K_{t+1})=\hyperref[alg:AC-GM]{\textbf{AC-GM}}(g(x^+,\cdot),z_t^{K_t},\alpha,L_{0,1},\epsilon_z)$.
\State $V^+=f(x^+,z_{t+1}^{K_{t+1}})$.
\While{$V^+>V_t-\|\widehat{\nabla}\varphi(x_t,y_t^{N_t},z_t^{K_t})\|^2/(4a)+T^{-1}$}\label{alg:bt:test}
\State $a\leftarrow2a$.\label{alg:bt:double}
\State $x^+=x_t-\widehat{\nabla}\varphi(x_t,y_t^{N_t},z_t^{K_t})/a$.
\State $(z_{t+1}^{K_{t+1}},K_{t+1})=\hyperref[alg:AC-GM]{\textbf{AC-GM}}(g(x^+,\cdot),z_t^{K_t},\alpha,L_{0,1},\epsilon_z)$.
\State $V^+=f(x^+,z_{t+1}^{K_{t+1}})$.
\EndWhile
\State Set $x_{t+1}=x^+$, $V_{t+1}=V^+$, and $a_t=a$.\label{alg:AF2BAxt+1}
\EndFor
\end{algorithmic}
\end{algorithm}

The allowance $T^{-1}$ in Line \ref{alg:bt:test} of Algorithm \ref{alg:AF2BA} accounts for errors in the hypergradient and hyperobjective values. Lemma \ref{lem:bt:value} bounds the value error under uniform P{\L}, while Lemma \ref{lem:hypererror} bounds the hypergradient error. Together with the smoothness of $\varphi$, these estimates ensure that the trial point is accepted after finitely many doublings of $a$.}

{
\begin{remark}
\label{rem:epsilon}
In Line \ref{alg:AF2BA:sigma} of Algorithm \ref{alg:AF2BA}, we set $\sigma = T^{-1}$ and $\epsilon_y=\epsilon_z=T^{-2}$. Throughout this paper, we assume that the prescribed horizon $T$ is sufficiently large to satisfy
\begin{equation}
\label{equ:bt:horizon}
T\ge\max\left\{1,\bar\sigma^{-1},\frac{\rho_f}{\rho_g},\frac{4l_f}{\mu}\right\}.
\end{equation}
This ensures that the penalty approximation is valid and that backtracking terminates. Additionally, the requirements in Lemmas \ref{lem:Ylips} and \ref{lem:nablalimit} are satisfied.
\end{remark}
}

\subsection{Approximate hypergradient}
The update step for the upper-level variable $x$ depends fundamentally on the exact hypergradient $\nabla \varphi(x)$ defined in \eqref{equ:lem:nablavarphi}. However, because computing the exact optimal solutions $y^*(x)$ and $y_{\sigma}^*(x)$ is computationally intractable, we must construct an approximation using the inexact solutions $\hat{z}$ and $\hat{y}$ obtained from the following two subproblems, respectively:
\begin{equation}
\label{p:lower}
\min_{z\in \sR^{d_y}}~g(x,z),
\end{equation}
and
\begin{equation}
\label{p:penalty}
\min_{y\in \sR^{d_y}}~ \sigma f(x,y) + g(x,y).
\end{equation}
Given these approximate solutions $\hat{z}$ and $\hat{y}$, the approximate hypergradient is defined as
\begin{equation}
\label{equ:apprhyper}
\widehat{\nabla} \varphi(x,\hat{y},\hat{z}) = \nabla_x f(x,\hat{y}) + \frac{\nabla_x g(x,\hat{y}) - \nabla_x g(x,\hat{z})}{\sigma}.
\end{equation}
Building upon Lemma \ref{lem:nablalimit}, the following lemma rigorously quantifies the deviation between the approximate hypergradient $\widehat{\nabla} \varphi$ and the exact hypergradient $\nabla \varphi$. The proof is deferred to Section \ref{sec:prooflem:hypererror}.
\begin{lemma}
\label{lem:hypererror}
Suppose that Assumption \ref{ass:basic} holds. Then, for any $x\in\sR^{d_x}$, given the approximate solutions $\hat{z}$ and $\hat{y}$ of Problems \eqref{p:lower} and \eqref{p:penalty}, it holds that
\begin{align}
\left\|\nabla\varphi(x) - \widehat{\nabla}\varphi(x,\hat{y},\hat{z})\right\| \le& \bar{D}_{\sigma} \sigma+ L_f \dist(\hat{y},Y_{\sigma}^*(x))\notag \\
&+ \frac{L_g}{\sigma}\dist(\hat{y},Y_{\sigma}^*(x)) + \frac{L_g}{\sigma}\dist(\hat{z},Y^*(x)),\notag
\end{align}
where $\bar{D}_{\sigma}$ is defined in Lemma \ref{lem:nablalimit}.
\end{lemma}

\subsection{Resolution of subproblems}
In this subsection, we formulate and analyze the Auto-Conditioned Gradient Method (\hyperref[alg:AC-GM]{AC-GM}), an adaptive subroutine for smooth problems satisfying the Polyak--{\L}ojasiewicz (P{\L}) condition. It is a smooth variant of the auto-conditioned proximal gradient method (AC-PGM) \citep{yagishita2025simple}, which extends the auto-conditioned fast gradient method (AC-FGM) of \cite{li2025simple}. We also give iteration bounds for applying \hyperref[alg:AC-GM]{AC-GM} to Problems \eqref{p:lower} and \eqref{p:penalty}.

\begin{algorithm}[htp]
\caption{\textbf{A}uto-\textbf{C}onditioned \textbf{G}radient \textbf{M}ethod:\\ $(x^k,k)=\text{\textbf{AC-GM}}(h,x^0,\alpha,L_0,\epsilon_h)$}
\label{alg:AC-GM}
\begin{algorithmic}[1]
\State Initial point $x^0$, scale parameter $\alpha > 1$, initial Lipschitz factor $L_0 > 0$, error tolerance $\epsilon_h$.
\State $k =0$.
\While{$\|\nabla h(x^k)\| > \epsilon_h$}\label{alg:AC-GM:stop}
\State $\gamma_{k + 1} = \max\{L_0, \dots, L_{k}\}$.
\State $x^{k+1} = x^{k} - \frac{1}{\alpha\gamma_{k+1}}\nabla h(x^{k})$.\label{alg:AC-GMXK+1}
\State $L_{k+1} = \frac{2(h(x^{k+1}) - h(x^{k}) - \langle \nabla h(x^{k}), x^{k+1} - x^{k} \rangle)}{\|x^{k+1} - x^{k}\|^2}$.\label{alg:AC-GMLk}
\State $k = k + 1$.
\EndWhile
\end{algorithmic}
\end{algorithm}

{AC-PGM has a nonconvex stationarity guarantee and a linear convergence guarantee for P{\L} problems \cite[Theorems 2.1 and 2.2]{yagishita2025simple}. Here, we derive an explicit linear bound for the objective gap of \hyperref[alg:AC-GM]{AC-GM} from the initial iterate, accounting for unsuccessful iterations. This explicit linear bound is not provided by the original AC-PGM analysis.} The proof is deferred to Section \ref{sec:prooflem:linear:ACGM}.
\begin{lemma}
\label{lem:linear:ACGM}
Given a $\mu_h$-P{\L} function $h :\sR^d \rightarrow \sR$ with $L_h$-Lipschitz gradients, the sequence $\{x^k\}$ generated by Algorithm \ref{alg:AC-GM} satisfies
\begin{equation*}
h(x^{k+1}) - h^* \le C_h\left(1-p\right)^{k+1}\left(h(x^{0}) - h^*\right),
\end{equation*}
where $p := \frac{\mu_h(\alpha -1)}{2\alpha^2 \max\{L_0,L_h\}}$ satisfies $0 < p < 1$ and $C_h$ is a constant defined in \eqref{equ:kinbarS6}.
\end{lemma}

The following proposition establishes the iteration complexity of Algorithm \ref{alg:AC-GM} when deployed to solve subproblems \eqref{p:lower} and \eqref{p:penalty}. The proof is deferred to Section \ref{sec:proofprop:KtNt2}.
\begin{proposition}
\label{prop:KtNt2}
Suppose that Assumption \ref{ass:basic} holds. Let $K_t$ and $N_t$ denote the numbers of \hyperref[alg:AC-GM]{AC-GM} iterations for the lower-level and penalty problems, respectively. Then
\[
\begin{aligned}
K_t\le\left\lceil\frac{[\log(\bar C_b\epsilon_z^{-2})]_+}{\log(1/(1-p_g))}\right\rceil,\quad N_t\le\left\lceil\frac{[\log(\bar C_c\epsilon_y^{-2})]_+}{\log(1/(1-p_\sigma))}\right\rceil,
\end{aligned}
\]
where $0<p_g,p_\sigma<1$ and $\bar C_b,\bar C_c>0$ are constants defined in the proof. The bound for $K_t$ also applies to each call at a rejected trial point in Algorithm \ref{alg:AF2BA}. Here $[u]_+:=\max\{u,0\}$.
\end{proposition}

Proposition \ref{prop:KtNt2} gives explicit logarithmic bounds on the first-order oracle queries needed to satisfy the residual stopping criteria for subproblems \eqref{p:lower} and \eqref{p:penalty}. This demonstrates that our adaptive subroutine matches the logarithmic iteration complexity results of well-tuned gradient methods \citep{attouch2009convergence, frankel2015splitting, karimi2016linear, bento2025convergence} under the P{\L} condition.

\subsection{Convergence analysis of Algorithm \ref{alg:AF2BA}}
In this subsection, we present the convergence result for Algorithm \ref{alg:AF2BA}, which bounds the average squared gradient norm and gives the corresponding iteration and first-order oracle complexities. The proof is deferred to Section \ref{sec:proofthm:conv}.
\begin{theorem}
\label{thm:conv}
Suppose that Assumptions \ref{ass:basic} and \ref{ass:inf} hold. Then, the sequence $\{x_t\}$ generated by Algorithm \ref{alg:AF2BA} satisfies
{
\begin{equation}
\label{equ:bt:mainbound}
\frac1T\sum_{t=0}^{T-1}\|\nabla\varphi(x_t)\|^2
\le\left(\frac{24(1+L_\varphi)}{T}+\frac{98\hat C}{T^2}\right)
\bigl(\varphi(x_0)-\varphi^*+1\bigr),
\end{equation}
}
where $\hat C$ is a constant defined in \eqref{equ:hatC}. That is, Algorithm \ref{alg:AF2BA} can generate an $\epsilon$-stationary point of Problem \eqref{p:hyper} after at most $T=\gO(\epsilon^{-2})$ iterations. Furthermore, the number of first-order oracle queries required by Algorithm \ref{alg:AF2BA} to generate such a point is $\widetilde{\gO}(\epsilon^{-2})$.
\end{theorem}

Theorem \ref{thm:conv} shows that our adaptive algorithm achieves an $\widetilde{\gO}(\epsilon^{-2})$ first-order oracle complexity under the uniform P{\L} condition. When ignoring the logarithmic factors and {condition-number dependence}, this bound aligns with the theoretical lower bound for smooth nonconvex optimization \citep{nesterov2018lectures} and nonconvex--strongly-convex bilevel optimization with first-order oracles \citep{chen2025condition,ji2026lower}.

\begin{remark}
D-TFBO, S-TFBO \citep{yang2025tuning}, and AdaRHD \citep{shi2025adaptive} also provide adaptive convergence guarantees. {These methods assume lower-level strong convexity, use second-order information to compute hypergradients, and employ AdaGrad-Norm-type stepsize strategies based on accumulated gradient norms \citep{ward2020adagrad}. AF${}^2$BA combines outer backtracking based on inexact hyperobjective values with AC-GM for the inner subproblems, using only function values and first derivatives under uniform P{\L}.} In terms of first-order oracle queries, it matches the $\widetilde{\gO}(\epsilon^{-2})$ complexity of S-TFBO and improves on the $\gO(\epsilon^{-4})$ bounds of D-TFBO and AdaRHD.
\end{remark}

{
Denote
\begin{equation}
\label{equ:condinum}
\ell:=\max\{1,l_f,L_f,L_g,\rho_g\} \text{ and } \kappa:=\frac{\ell}{\mu}.
\end{equation}
The following corollary makes the condition-number dependence explicit.
\begin{corollary}
\label{cor:condition}
Under the same settings as in Theorem \ref{thm:conv}, suppose that $L_{0,1},L_{0,2}=\Theta(\ell)$, $\varphi(x_0)-\varphi^*=\gO(1)$, and $\dist(y_0,Y^*(x_0)),\dist(z_0,Y^*(x_0))=\gO(1)$. Then, Algorithm \ref{alg:AF2BA} generates an $\epsilon$-stationary point with the following first-order oracle complexity:
\begin{equation}
\label{equ:bt:queryrate}
\widetilde{\gO}\left(\ell\kappa^4\epsilon^{-2}\right).
\end{equation}
\end{corollary}
Corollary \ref{cor:condition} demonstrates that AF${}^2$BA attains the same condition-number dependence in its first-order oracle complexity as F${}^2$BA \citep{chen2024finding}, while selecting its stepsizes without prior knowledge of problem-specific constants, including the smoothness and P{\L} constants.
}

\section{Theoretical Proofs}
In this section, we prove the convergence and complexity results stated in Section \ref{sec:Ada}.

\subsection{Preparatory Lemmas}
We first provide the auxiliary lemmas used in the convergence and complexity analysis.

\begin{lemma}
\label{lem:pointdist}
Suppose that Assumption \ref{ass:basic} holds. Let $z$ and $y$ be the outputs of \hyperref[alg:AC-GM]{AC-GM} applied to $g(x,\cdot)$ and $g_\sigma(x,\cdot)$, respectively. Then
\[
\dist(z,Y^*(x))\le\epsilon_z/\mu,\qquad
\dist(y,Y_\sigma^*(x))\le\epsilon_y/\mu.
\]
The bound for $z$ also holds for calls made at backtracking trial points.
\end{lemma}
\begin{proof}
The result follows from the residual bounds at termination of the while loop in Line \ref{alg:AC-GM:stop} of Algorithm \ref{alg:AC-GM} and Lemma \ref{lem:PL}.
\end{proof}

\begin{lemma}
\label{lem:boundhyperg}
Suppose that Assumption \ref{ass:basic} holds. Then
\[
\|\widehat{\nabla}\varphi(x_t,y_t^{N_t},z_t^{K_t})\|\le C_\varphi,
\]
where $C_\varphi$ is a constant.
\end{lemma}
\begin{proof}
By Lemmas \ref{lem:hypererror} and \ref{lem:pointdist}, and since $\sigma=T^{-1}\le1$, we have
\[
\left\|\widehat{\nabla}\varphi(x_t,y_t^{N_t},z_t^{K_t})-\nabla\varphi(x_t)\right\|
\le \bar D_\sigma+\frac{L_f+2L_g}{\mu}.
\]
Lemma \ref{lem:nablavarphi} gives $\|\nabla\varphi(x_t)\|\le l_f\left(1+\frac{L_g}{\mu}\right)$. Thus, by the triangle inequality,
\begin{equation}
\label{equ:Cphi}
\left\|\widehat{\nabla}\varphi(x_t,y_t^{N_t},z_t^{K_t})\right\|
\le l_f\left(1+\frac{L_g}{\mu}\right)+\bar D_\sigma+\frac{L_f+2L_g}{\mu}
:=C_\varphi.
\end{equation}
The proof is complete.
\end{proof}

{The following lemma bounds the error between the computed upper-level value $f(x,z)$ and the hyperobjective $\varphi(x)$ under uniform P{\L}. This bound is used to justify the outer backtracking condition.
\begin{lemma}[Approximate hyperobjective values]
\label{lem:bt:value}
Under Assumption \ref{ass:basic}, every $z^*\in Y^*(x)$ satisfies $f(x,z^*)=\varphi(x)$. If $\|\nabla_y g(x,z)\|\le\epsilon_z$, then
\begin{equation}
\label{equ:bt:valueerror}
|f(x,z)-\varphi(x)|\le\frac{l_f\epsilon_z}{\mu}.
\end{equation}
\end{lemma}
\begin{proof}
For $z_1,z_2\in Y^*(x)$, we have $g(x,z_1)=g(x,z_2)$ and $\nabla_y g(x,z_1)=0$. Hence, for $0<\sigma\le\bar\sigma$, the uniform P{\L} condition gives
\begin{align*}
\sigma\bigl(f(x,z_1)-f(x,z_2)\bigr)
&=g_\sigma(x,z_1)-g_\sigma(x,z_2)\\
&\le g_\sigma(x,z_1)-\min_y g_\sigma(x,y)\\
&\le\frac{\|\nabla_y g_\sigma(x,z_1)\|^2}{2\mu}\\
&=\frac{\sigma^2\|\nabla_y f(x,z_1)\|^2}{2\mu}
\le\frac{\sigma^2l_f^2}{2\mu}.
\end{align*}
Interchanging $z_1$ and $z_2$ gives the reverse inequality, so
\[
|f(x,z_1)-f(x,z_2)|\le\frac{\sigma l_f^2}{2\mu}.
\]
Letting $\sigma\downarrow0$, it holds that $f(x,\cdot)$ is constant on $Y^*(x)$, with value $\varphi(x)$. Finally, Lipschitz continuity of $f$ and Lemma \ref{lem:PL} give
\[
|f(x,z)-\varphi(x)|\le l_f\dist(z,Y^*(x))\le\frac{l_f\epsilon_z}{\mu}.
\]
The proof is complete.
\end{proof}
}

\subsection{Proof of Lemma \ref{lem:hypererror}}
\label{sec:prooflem:hypererror}
\begin{proof}
We first bound the distance between $\widehat{\nabla} \varphi(x,\hat{y},\hat{z})$ and $\nabla\varphi_\sigma(x)$. {Since the solution sets are nonempty and closed, choose}
\[
z^*(x)\in\argmin_{z\in Y^*(x)}\|\hat z-z\|,\qquad
y_\sigma^*(x)\in\argmin_{y\in Y_\sigma^*(x)}\|\hat y-y\|.
\]
By the expression for $\nabla\varphi_\sigma(x)$ in Lemma \ref{lem:nablalimit}, we have
\begin{align}
&\|\widehat{\nabla} \varphi(x,\hat{y},\hat{z}) - \nabla\varphi_\sigma(x)\| \notag \\
\le& \|\nabla_x f(x,\hat{y}) - \nabla_x f(x,y_{\sigma}^*(x))\| + \frac{\|\nabla_x g(x,\hat{y}) - \nabla_x g(x,y_{\sigma}^*(x))\|}{\sigma} \notag \\
&+ \frac{\|\nabla_x g(x,\hat{z}) - \nabla_x g(x,z^*(x))\|}{\sigma}\notag\\
\le & L_f \dist(\hat{y}, Y_{\sigma}^*(x)) + \frac{L_g}{\sigma}\dist(\hat{y},Y_{\sigma}^*(x)) + \frac{L_g}{\sigma}\dist(\hat{z},Y^*(x)),\notag
\end{align}
where the second inequality follows from Assumption \ref{ass:basic}.

Then, by the triangle inequality and Lemma \ref{lem:nablalimit}, we have
\begin{align}
&\left\|\nabla\varphi(x) - \widehat{\nabla}\varphi(x, \hat{y}, \hat{z})\right\|\notag\\ 
\le& \|\nabla\varphi(x) - \nabla\varphi_\sigma(x)\| + \|\nabla\varphi_\sigma(x) - \widehat{\nabla}\varphi(x, \hat{y}, \hat{z})\| \notag \\
\le & \bar{D}_{\sigma} \sigma + \left(L_f \dist(\hat{y},Y_{\sigma}^*(x)) + \frac{L_g}{\sigma}\dist(\hat{y},Y_{\sigma}^*(x)) + \frac{L_g}{\sigma}\dist(\hat{z},Y^*(x))\right),\notag
\end{align}
where the second inequality follows from Lemma \ref{lem:nablalimit} and the preceding inequality.
\end{proof}

\subsection{Proof of Lemma \ref{lem:linear:ACGM}}
\label{sec:prooflem:linear:ACGM}
Denote
\begin{equation}
\label{equ:SbarS}
\gS := \left\{k \ge 0 \mid \beta \gamma_{k+1} \ge L_{k+1} \right\},\quad
\bar{\gS} := \{0,1,\dots\}\setminus\gS,
\end{equation}
where $\beta := \frac{\alpha+1}{2} > 1$, $\gamma_{k+1} > 0$ and $L_{k+1}$ are defined in Algorithm \ref{alg:AC-GM}.

Prior to proving Lemma \ref{lem:linear:ACGM}, we recall a smooth variant of Lemma 2.1 in \cite{yagishita2025simple} alongside its proof.
\begin{lemma}
\label{lem:yagi}
Given a $\mu_h$-P{\L} function $h :\sR^d \rightarrow \sR$ with $L_h$-Lipschitz gradients, the sequence $\{x^k\}$ generated by Algorithm \ref{alg:AC-GM} satisfies
\begin{equation}
\label{equ:lem:yagi}
\frac{\alpha-1}{4\alpha^2}\sum_{i = 0}^{k}\frac{1}{\gamma_{i+1}}\|\nabla h(x^{i})\|^2 \le h(x^0) - h(x^{k+1}) + \sum_{i \in [k] \cap \bar{\gS}}\frac{\gamma_{i+2}-\gamma_{i+1}}{2}\|x^{i+1}-x^{i}\|^2,
\end{equation}
where $\bar{\gS}$ is defined in \eqref{equ:SbarS} and $[k]$ denotes $\{0,\dots,k\}$.

Furthermore, the number of elements in the set $\bar{\gS}$ satisfies
\begin{equation*}
|\bar{\gS}| \le \left\lceil \log_\beta\frac{\max\left\{L_0,L_h\right\}}{L_0} \right\rceil_+ + 1 := m_h,
\end{equation*}
where $\lceil a \rceil_{+}$ represents the smallest nonnegative integer greater than or equal to $a$.
\end{lemma}
\begin{proof}
By the update rule of $x^{k+1}$ in Line \ref{alg:AC-GMXK+1}, it holds that
\begin{align}
&\langle \nabla h(x^{k}), x^{k+1} - x^{k}\rangle + \frac{\alpha\gamma_{k+1}}{2}\|x^{k+1} - x^{k}\|^2\notag\\
=& -\frac{1}{\alpha\gamma_{k+1}}\|\nabla h(x^{k})\|^2 + \frac{1}{2\alpha\gamma_{k+1}}\|\nabla h(x^{k})\|^2 \le 0.\notag
\end{align}
Then, by the definition of $L_{k+1}$ in Line \ref{alg:AC-GMLk} of Algorithm \ref{alg:AC-GM}, we have
\begin{align}
\label{equ:lem:yagi1}
&\frac{\alpha\gamma_{k+1} - L_{k+1}}{2}\|x^{k+1} - x^{k}\|^2 + h(x^{k+1}) - h(x^{k})\notag\\
= & \frac{\alpha\gamma_{k+1}}{2}\|x^{k+1} - x^{k}\|^2 - \frac{h(x^{k+1}) - h(x^{k}) - \langle \nabla h(x^{k}), x^{k+1} - x^{k} \rangle}{\|x^{k+1} - x^{k}\|^2}\|x^{k+1} - x^{k}\|^2 \notag\\
&+ h(x^{k+1}) - h(x^{k})\notag\\
= & \frac{\alpha\gamma_{k+1}}{2}\|x^{k+1} - x^{k}\|^2 + \langle \nabla h(x^{k}), x^{k+1} - x^{k} \rangle \le 0.
\end{align}
If $k \in \gS$, it holds that $\beta\gamma_{k+1} - L_{k+1} \ge 0$. Then, we have
\begin{align}
\label{equ:lem:yagi2}
h(x^{k}) - h(x^{k+1}) &\overset{\eqref{equ:lem:yagi1}}{\ge} \frac{\alpha\gamma_{k+1} - L_{k+1}}{2}\|x^{k+1} - x^{k}\|^2\notag\\
\ge& \frac{\alpha\gamma_{k+1} - \frac{\alpha + 1}{2}\gamma_{k+1}}{2}\|x^{k+1} - x^{k}\|^2\notag\\
\ge& \frac{\alpha-1}{4}\gamma_{k+1}\|x^{k+1} - x^{k}\|^2\notag\\
=& \frac{\alpha-1}{4\alpha^2\gamma_{k+1}}\|\nabla h(x^{k})\|^2,
\end{align}
where the second inequality follows from the definition $\beta = (\alpha + 1) / 2$, and the equality follows from the update rule of $x^{k+1}$.

On the other hand, if $k \in \bar{\gS}$, it holds that $\beta\gamma_{k+1} \le L_{k+1}$, which implies $\gamma_{k+1} \le L_{k+1}$ as $\beta \ge 1$, and therefore, we have
\begin{equation}
\label{equ:lem:yagigammak+2}
\gamma_{k+2} = \max\{L_0,\dots,L_{k+1}\} = L_{k+1}.
\end{equation}
Then, by the update rule of $x^{k+1}$, we have
\begin{align}
\label{equ:lem:yagi3}
&\frac{\alpha-1}{4\alpha^2\gamma_{k+1}}\|\nabla h(x^{k})\|^2\notag\\
\le &\frac{\alpha-1}{2\alpha^2\gamma_{k+1}}\|\nabla h(x^{k})\|^2 = \frac{\alpha-1}{2}\gamma_{k+1}\|x^{k+1} - x^{k}\|^2 \notag\\
=& \frac{\alpha \gamma_{k+1} - L_{k+1}}{2}\|x^{k+1} - x^{k}\|^2 + \frac{L_{k+1} - \gamma_{k+1}}{2}\|x^{k+1} - x^{k}\|^2 \notag\\
\le& h(x^{k}) - h(x^{k+1}) + \frac{\gamma_{k+2} - \gamma_{k+1}}{2}\|x^{k+1} - x^{k}\|^2,
\end{align}
where the last inequality follows from \eqref{equ:lem:yagi1}, \eqref{equ:lem:yagigammak+2}, and $\gamma_{k+1} \le L_{k+1}$.

Summing \eqref{equ:lem:yagi2} and \eqref{equ:lem:yagi3} from $0$ to $k$, we have
\begin{equation*}
\frac{\alpha-1}{4\alpha^2}\sum_{i = 0}^{k}\frac{1}{\gamma_{i+1}}\|\nabla h(x^{i})\|^2 \le h(x^0) - h(x^{k+1}) + \sum_{i \in [k] \cap \bar{\gS}}\frac{\gamma_{i+2}-\gamma_{i+1}}{2}\|x^{i+1}-x^{i}\|^2.
\end{equation*}
The proof of \eqref{equ:lem:yagi} is complete.

Denote $\bar{\gS} = \{k_0,\dots,k_m\}$. By the definitions of $\gamma_{k+1}$ and $L_{k+1}$, it holds that $L_0 \le \gamma_{k+1} \le \max\{L_0,L_h\}$ for all $k \ge 0$. For any $k_j \in \bar{\gS}$, we further have $\beta \gamma_{k_j+1} \le L_{k_j+1} \le \max\{L_0,L_h\}$. Then, from \eqref{equ:lem:yagigammak+2}, it holds that
\begin{align}
\max\{L_0,L_h\} \ge & \gamma_{k_m + 1} \ge \gamma_{k_{m - 1} + 2} \overset{\eqref{equ:lem:yagigammak+2}}{=} L_{k_{m - 1} + 1}\notag\\
\ge & \beta \gamma_{k_{m - 1} + 1} \ge \beta^m \gamma_{k_{0} + 1} \ge \beta^m L_0,\notag
\end{align}
where we use the monotonicity of $\gamma_{k+1}$, the definition of unsuccessful iterations, and the fact that $\gamma_{k_0+1}\ge L_0$. The desired result follows.
\end{proof}

Building upon Lemma \ref{lem:yagi}, we now provide the proof of Lemma \ref{lem:linear:ACGM}.
\begin{proof}[Lemma \ref{lem:linear:ACGM}]
If $k\in\gS$, by \eqref{equ:lem:yagi2} and the fact that $\gamma_{k+1} \le \max\{L_0,L_h\}$, we have
\begin{align}
\label{equ:linear}
\frac{\alpha-1}{4\alpha^2 \max\{L_0, L_h\}} \|\nabla h(x^k)\|^2 \le& \frac{\alpha-1}{4\alpha^2\gamma_{k+1}}\|\nabla h(x^{k})\|^2 \notag \\
\overset{\eqref{equ:lem:yagi2}}{\le}& h(x^{k}) -h(x^{k+1})\notag\\
= & (h(x^{k}) - h^*) - (h(x^{k+1})-h^*).
\end{align}
Rearranging \eqref{equ:linear} and using the $\mu_h$-P{\L} condition of $h$, it holds that
\begin{equation*}
 p(h(x^{k})-h^*)\le\frac{\mu_h(\alpha -1)}{2\alpha^2 \max\{L_0,L_h\}}(h(x^{k})-h^*)\le (h(x^{k}) - h^*)-(h(x^{k+1})-h^*),
\end{equation*}
which implies
\begin{align}
\label{equ:upperS}
h(x^{k+1}) - h^* \le & \left(1 -p\right)\left(h(x^{k}) - h^*\right).
\end{align}

On the other hand, if $k \in \bar{\gS}$, it holds that $\gamma_{k+2} = L_{k+1}$. By the last equality and inequality in \eqref{equ:lem:yagi3}, we have
\begin{align}
\left(\frac{\alpha -1}{2\alpha^2 \gamma_{k+1}} - \frac{\gamma_{k+2}-\gamma_{k+1}}{2\alpha^2\gamma_{k+1}^2}\right)\|\nabla h(x^{k})\|^2 =& \frac{\alpha -1}{2\alpha^2 \gamma_{k+1}}\|\nabla h(x^{k})\|^2\notag\\
&- \frac{\gamma_{k+2}-\gamma_{k+1}}{2}\|x^{k+1} - x^{k}\|^2\notag\\
\le& h(x^{k})-h(x^{k+1}),\notag
\end{align}
which is equivalent to
\begin{equation}
\label{equ:kinbarS}
\frac{\alpha \gamma_{k+1} -\gamma_{k+2}}{2\alpha^2 \gamma_{k+1}^2}\|\nabla h(x^{k})\|^2 \le h(x^{k}) -h(x^{k+1}) = h(x^{k}) - h^* - (h(x^{k+1}) - h^*).
\end{equation}
If $\alpha \gamma_{k+1} -\gamma_{k+2} < 0$, by Lemma \ref{lem:PL}, we have
\begin{align}
\label{equ:PLprop}
\frac{\alpha \gamma_{k+1} -\gamma_{k+2}}{2\alpha^2 \gamma_{k+1}^2}\|\nabla h(x^{k})\|^2 \ge& \frac{L_h^2(\alpha \gamma_{k+1} -\gamma_{k+2})}{2\alpha^2 \gamma_{k+1}^2}\dist^2(x^k,X_h^*)\notag\\
\ge& \frac{L_h^2(\alpha \gamma_{k+1} -\gamma_{k+2})}{\mu_h\alpha^2 \gamma_{k+1}^2}\left(h(x^{k})-h^*\right).
\end{align}
Substituting \eqref{equ:PLprop} into \eqref{equ:kinbarS}, we have
\begin{equation*}
\frac{L_h^2(\alpha\gamma_{k+1}-\gamma_{k+2})}{\mu_h\alpha^2\gamma_{k+1}^2}\left(h(x^{k})-h^*\right) \le h(x^{k}) - h^* - (h(x^{k+1}) - h^*),
\end{equation*}
which is equivalent to
\begin{align}
\label{equ:kinbarS2}
h(x^{k+1})-h^* \le& \left(1-\frac{L_h^2(\alpha\gamma_{k+1} -\gamma_{k+2})}{\mu_h\alpha^2\gamma_{k + 1}^2}\right)(h(x^{k})-h^*)\notag\\
\le& \left(1+\frac{L_h^2\gamma_{k+2}}{\mu_h \alpha^2 \gamma_{k+1}^2}\right) (h(x^{k})-h^*).
\end{align}
If $\alpha \gamma_{k+1} -\gamma_{k+2} \ge 0$, plugging the P{\L} condition \eqref{equ:def:PL} into \eqref{equ:kinbarS} yields
\begin{align}
\frac{\mu_h(\alpha \gamma_{k+1} -\gamma_{k+2})}{\alpha^2\gamma_{k+1}^2}\left(h(x^{k})-h^*\right) \le& \frac{\alpha \gamma_{k+1} -\gamma_{k+2}}{2\alpha^2 \gamma_{k+1}^2}\|\nabla h(x^{k})\|^2\notag\\
\le& h(x^{k}) - h^* - (h(x^{k+1}) - h^*),\notag
\end{align}
which is equivalent to
\begin{align}
\label{equ:kinbarS3}
h(x^{k+1})-h^*\le& \left(1-\frac{\mu_h(\alpha \gamma_{k+1} -\gamma_{k+2})}{\alpha^2\gamma_{k+1}^2}\right)\left(h(x^{k})-h^*\right)\notag\\ 
\le& \left(1+\frac{\mu_h\gamma_{k+2}}{\alpha^2\gamma_{k+1}^2}\right)\left(h(x^{k})-h^*\right).
\end{align}
Therefore, combining \eqref{equ:kinbarS2} and \eqref{equ:kinbarS3}, we have the following upper bound on $h(x^{k+1})-h^*$ for all $k\in \bar{\gS}$:
\begin{align}
\label{equ:kinbarS4}
&h(x^{k+1})-h^*\notag\\
\le& \left(1+\max \left(\frac{L_h^2\gamma_{k+2}}{\mu_h \alpha^2 \gamma_{k+1}^2},\frac{\mu_h\gamma_{k+2}}{\alpha^2\gamma_{k+1}^2}\right)\right)\left(h(x^{k})-h^*\right)\notag\\
\le& \left(1 + \max\left\{\frac{L_h^3}{\mu_h \alpha^2 L_0^2},\frac{L_h^2}{\mu_h \alpha^2 L_0},\frac{\mu_h L_h}{\alpha^2 L_0^2},\frac{\mu_h}{\alpha^2 L_0}\right\}\right)\left(h(x^{k})-h^*\right)\notag\\
:=& \left(1 + \bar{C}\right)\left(h(x^{k})-h^*\right),
\end{align}
where the second inequality follows from $L_0\le\gamma_{k+1},\gamma_{k+2}\le\max\{L_0,L_h\}$.

Let $r_k$ denote the number of unsuccessful iterations among the first $k+1$ updates. By \eqref{equ:upperS} and \eqref{equ:kinbarS4}, each successful update multiplies the objective gap by at most $1-p$, and each unsuccessful update multiplies it by at most $1+\bar C$.
\begin{equation*}
h(x^{k+1})-h^*\le(1+\bar C)^{r_k}(1-p)^{k+1-r_k}(h(x^0)-h^*).
\end{equation*}

By Lemma \ref{lem:yagi}, $r_k\le m_h$. Since $(1+\bar C)/(1-p)>1$, this implies
\begin{align}
\label{equ:kinbarS6}
h(x^{k+1})-h^*
&\le\left(\frac{1+\bar C}{1-p}\right)^{r_k}(1-p)^{k+1}(h(x^0)-h^*)\\
&\le\left(\frac{1+\bar C}{1-p}\right)^{m_h}(1-p)^{k+1}(h(x^0)-h^*)\\
&=C_h(1-p)^{k+1}(h(x^0)-h^*),
\end{align}
where $C_h:=((1+\bar C)/(1-p))^{m_h}$. This proves the claim.
\end{proof}

\subsection{{Proof of Proposition \ref{prop:KtNt2}}}
\label{sec:proofprop:KtNt2}
\begin{proof}
We first derive a uniform residual bound for \hyperref[alg:AC-GM]{AC-GM} that applies to both subproblems. Let $h$ denote either $g(x,\cdot)$ or $g_\sigma(x,\cdot)$, and denote the corresponding inner iterates by $x_h^j$. We take $L_h=L_g$ for the lower-level objective $g$ and $L_h=L_f+L_g$ for the penalty objective $g_\sigma$.

By Lemma \ref{lem:linear:ACGM}, it holds that
\[
h(x_h^j)-h^*\le C_h(1-p_h)^j\bigl(h(x_h^0)-h^*\bigr),
\]
where $L_{0,h}=L_{0,1}$ for the lower-level problem and $L_{0,h}=L_{0,2}$ for the penalty problem. We take
\[
p_h=\frac{\mu}{8\max\{L_{0,h},L_h\}},\qquad
m_h=\left\lceil\log_{3/2}\frac{\max\{L_{0,h},L_h\}}{L_{0,h}}\right\rceil+1,
\]
and
\[
\bar C_h=\max\left\{\frac{L_h^3}{4\mu L_{0,h}^2},\frac{L_h^2}{4\mu L_{0,h}},\frac{\mu L_h}{4L_{0,h}^2},\frac{\mu}{4L_{0,h}}\right\},
\qquad
C_h=\left(\frac{1+\bar C_h}{1-p_h}\right)^{m_h}.
\]
We denote these constants by $p_g,C_g$ when $L_h=L_g$, and by $p_\sigma,C_\sigma$ when $L_h=L_f+L_g$.

By Lemma \ref{lem:PL}, we have
\begin{equation}
\label{equ:prop:ACGMresidual}
\begin{aligned}
\|\nabla h(x_h^j)\|^2
&\le L_h^2\dist^2(x_h^j,X_h^*)\\
&\le\frac{2L_h^2}{\mu}\bigl(h(x_h^j)-h^*\bigr)\\
&\le\frac{2L_h^2C_h}{\mu}(1-p_h)^j\bigl(h(x_h^0)-h^*\bigr)\\
&\le\frac{L_h^3C_h}{\mu}(1-p_h)^j\dist^2(x_h^0,X_h^*),
\end{aligned}
\end{equation}
where the first two inequalities follow from Lemma \ref{lem:PL}, the third follows from Lemma \ref{lem:linear:ACGM}, and the last follows from
\[
h(x_h^0)-h^*\le\frac{L_h}{2}\dist^2(x_h^0,X_h^*).
\]
We next bound the initial distances for the lower-level problem. At initialization, \hyperref[alg:AC-GM]{AC-GM} is applied to $g(x_0,\cdot)$ with initial point $z_0$; at each outer backtracking trial point $x^+$, it is applied to $g(x^+,\cdot)$ with initial point $z_t^{K_t}$. 

By Lemmas \ref{lem:PL} and \ref{lem:Ylips}, we have
\begin{equation}
\label{equ:trialdist}
\begin{aligned}
\dist(z_t^{K_t},Y^*(x^+))
&\le\dist(z_t^{K_t},Y^*(x_t))
+\dist_H(Y^*(x_t),Y^*(x^+))\\
&\le\frac{\|\nabla_y g(x_t,z_t^{K_t})\|}{\mu}
+\frac{L_g}{\mu}\|x^+-x_t\|\\
&\le\frac{\epsilon_z}{\mu}+\frac{L_g}{\mu}\|x^+-x_t\|\\
&\le\frac{1+L_gC_\varphi}{\mu},
\end{aligned}
\end{equation}
where the first inequality follows from the triangle inequality, the second from Lemmas \ref{lem:PL} and \ref{lem:Ylips}, the third from the residual stopping condition at $x_t$, and the last from Lemma \ref{lem:boundhyperg} and $a \ge 1$. This bound applies to every trial in Line \ref{alg:bt:test} of Algorithm \ref{alg:AF2BA}, whether the trial point is accepted or rejected.

Let $z^j$ denote the inner iterates for the lower-level problem. The two possible initializations are $x=x_0$, $z^0=z_0$ and $x=x^+$, $z^0=z_t^{K_t}$; for the latter, \eqref{equ:trialdist} bounds the initial distance to $Y^*(x^+)$. Applying \eqref{equ:prop:ACGMresidual} to both cases and taking the larger bound, we obtain
\begin{equation}
\label{equ:barCb}
\begin{aligned}
\|\nabla_y g(x,z^j)\|^2
\le\frac{L_g^3C_g}{\mu}(1-p_g)^j\dist^2(z^0,Y^*(x)) \le\bar C_b(1-p_g)^j,
\end{aligned}
\end{equation}
where
\[
\begin{aligned}
\bar C_b:=\max\bigg\{1,\frac{L_g^3C_g}{\mu}\dist^2(z_0,Y^*(x_0)), \frac{L_g^3C_g}{\mu^3}(1+L_gC_\varphi)^2\bigg\}.
\end{aligned}
\]
The second entry of $\bar C_b$ bounds the initial solve from $z_0$, and the third bounds each trial solve by \eqref{equ:trialdist}.

We then bound the initial distances for the penalty problem. When $t=0$, the initial point is $y_0$. By Lemma \ref{lem:Ylips}, we have
\begin{align*}
\dist(y_0,Y_\sigma^*(x_0))
&\le\dist(y_0,Y^*(x_0))
+\dist_H(Y^*(x_0),Y_\sigma^*(x_0))\\
&\le\dist(y_0,Y^*(x_0))+\frac{l_f\sigma}{\mu}\\
&\le\dist(y_0,Y^*(x_0))+\frac{l_f}{\mu},
\end{align*}
where the first inequality follows from the triangle inequality, and the last two follow from Lemma \ref{lem:Ylips} and $\sigma\le1$, respectively.

For $t\ge1$, the initial point is $y_{t-1}^{N_{t-1}}$, which satisfies the residual stopping condition at $x_{t-1}$. Thus,
\begin{equation}
\label{equ:penaltywarmstartdist}
\begin{aligned}
&\dist(y_{t-1}^{N_{t-1}},Y_\sigma^*(x_t))\\
&\le\dist(y_{t-1}^{N_{t-1}},Y_\sigma^*(x_{t-1}))
+\dist_H(Y_\sigma^*(x_{t-1}),Y_\sigma^*(x_t))\\
&\le\frac{\|\nabla_y g_\sigma(x_{t-1},y_{t-1}^{N_{t-1}})\|}{\mu}
+\frac{\sigma L_f+L_g}{\mu}\|x_t-x_{t-1}\|\\
&\le\frac{\epsilon_y}{\mu}
+\frac{\sigma L_f+L_g}{\mu}\|x_t-x_{t-1}\|\\
&\le\frac{1+(L_f+L_g)C_\varphi}{\mu},
\end{aligned}
\end{equation}
where the first two inequalities follow from the triangle inequality and Lemmas \ref{lem:PL} and \ref{lem:Ylips}, the third uses the residual stopping condition at $x_{t-1}$, and the last follows from Lemma \ref{lem:boundhyperg} and $a_{t-1} \ge1 $.

Let $y_t^j$ denote the inner iterates for the penalty problem at $x_t$, with $y_t^0=y_0$ for $t=0$ and $y_t^0=y_{t-1}^{N_{t-1}}$ for $t\ge1$. Combining the preceding estimates for these two cases with \eqref{equ:prop:ACGMresidual}, we obtain
\begin{equation}
\label{equ:barCc}
\begin{aligned}
&\|\nabla_y g_\sigma(x_t,y_t^j)\|^2\\
&\le\frac{(L_f+L_g)^3C_\sigma}{\mu}(1-p_\sigma)^j
\dist^2(y_t^0,Y_\sigma^*(x_t))\\
&\le\bar C_c(1-p_\sigma)^j,
\end{aligned}
\end{equation}
where
\[
\begin{aligned}
\bar C_c:=\max\bigg\{&1,\frac{(L_f+L_g)^3C_\sigma}{\mu}
\left(\dist(y_0,Y^*(x_0))+\frac{l_f}{\mu}\right)^2,\\
&\frac{(L_f+L_g)^3C_\sigma}{\mu^3}
\bigl(1+(L_f+L_g)C_\varphi\bigr)^2\bigg\}.
\end{aligned}
\]
The second entry of $\bar C_c$ bounds the initial penalty problem solve, and the third bounds the later solves by \eqref{equ:penaltywarmstartdist}. The entry $1$ ensures $\bar C_c\ge1$. Thus, by \eqref{equ:barCb} and \eqref{equ:barCc}, the while condition in Line \ref{alg:AC-GM:stop} of Algorithm \ref{alg:AC-GM} is false whenever
\[
\bar C_b(1-p_g)^K\le\epsilon_z^2,\qquad
\bar C_c(1-p_\sigma)^N\le\epsilon_y^2.
\]
Since $0<p_g,p_\sigma<1$, taking logarithms gives
\[
K\log\frac1{1-p_g}\ge\log\frac{\bar C_b}{\epsilon_z^2},
\qquad
N\log\frac1{1-p_\sigma}\ge\log\frac{\bar C_c}{\epsilon_y^2}.
\]
Consequently, we have
\[
\begin{aligned}
K_t\le\left\lceil\frac{[\log(\bar C_b\epsilon_z^{-2})]_+}{\log(1/(1-p_g))}\right\rceil, \qquad N_t\le\left\lceil\frac{[\log(\bar C_c\epsilon_y^{-2})]_+}{\log(1/(1-p_\sigma))}\right\rceil.
\end{aligned}
\]
In particular, the upper bound on $K_t$ also applies to every \hyperref[alg:AC-GM]{AC-GM} call for the lower-level problem, including initialization and rejected trials. The proof is complete.
\end{proof}

\subsection{{Proof of Theorem \ref{thm:conv}}}
\label{sec:proofthm:conv}
\begin{proof}
The residual tolerances in Algorithm \ref{alg:AF2BA} are $\epsilon_y=\epsilon_z=T^{-2}$, and $\sigma=T^{-1}$. Lemmas \ref{lem:hypererror} and \ref{lem:pointdist} therefore give
\begin{align*}
\left\|\widehat{\nabla}\varphi(x_t,y_t^{N_t},z_t^{K_t})-\nabla\varphi(x_t)\right\|
&\le \bar D_\sigma\sigma
 +\left(L_f+\frac{L_g}{\sigma}\right)\dist(y_t^{N_t},Y_\sigma^*(x_t))\\
&\quad +\frac{L_g}{\sigma}\dist(z_t^{K_t},Y^*(x_t))\\
&\le\frac{\bar D_\sigma}{T}+\frac{L_f}{\mu T^2}
+\frac{2L_g}{\mu T}.
\end{align*}
Since $T\ge1$, the Cauchy--Schwarz inequality yields
\begin{equation}
\label{equ:hatC}
\left\|\widehat{\nabla}\varphi(x_t,y_t^{N_t},z_t^{K_t})
-\nabla\varphi(x_t)\right\|^2
\le\frac{\hat C}{T^2},
\end{equation}
where
\[
\hat C:=3\bar D_\sigma^2+\frac{3L_f^2+12L_g^2}{\mu^2}.
\]
We next show that the while loop in Line \ref{alg:bt:test} of Algorithm \ref{alg:AF2BA} terminates after finitely many iterations. For a trial point
\[
x^+=x_t-\frac{\widehat{\nabla}\varphi(x_t,y_t^{N_t},z_t^{K_t})}{a}
\]
with $a\ge L_\varphi$, the smoothness of $\varphi$ gives
\begin{equation}
\label{equ:bt:descent}
\begin{aligned}
\varphi(x^+)
&\le\varphi(x_t)+\langle\nabla\varphi(x_t),x^+-x_t\rangle
+\frac{L_\varphi}{2}\|x^+-x_t\|^2\\
&=\varphi(x_t)-\frac{\|\widehat{\nabla}\varphi(x_t,y_t^{N_t},z_t^{K_t})\|^2}{a}\\
&\quad+\frac{\left\langle\widehat{\nabla}\varphi(x_t,y_t^{N_t},z_t^{K_t})-\nabla\varphi(x_t),
\widehat{\nabla}\varphi(x_t,y_t^{N_t},z_t^{K_t})\right\rangle}{a}\\
&\quad+\frac{L_\varphi\|\widehat{\nabla}\varphi(x_t,y_t^{N_t},z_t^{K_t})\|^2}{2a^2}\\
&\le\varphi(x_t)-\frac{\|\widehat{\nabla}\varphi(x_t,y_t^{N_t},z_t^{K_t})\|^2}{4a}
+\frac{\hat C}{aT^2},
\end{aligned}
\end{equation}
where the equality follows from the definition of $x^+$, and the last inequality follows from $a\ge L_\varphi$, \eqref{equ:hatC}, and Young's inequality $\langle u,v\rangle\le\|u\|^2+\frac{1}{4}\|v\|^2$.

Recall that the function values used in Line \ref{alg:bt:test} of Algorithm \ref{alg:AF2BA} are
\[
V_t=f(x_t,z_t^{K_t}),\qquad V^+=f(x^+,z_{t+1}^{K_{t+1}}),
\]
where $z_t^{K_t}$ and $z_{t+1}^{K_{t+1}}$ satisfy
\[
\|\nabla_y g(x_t,z_t^{K_t})\|\le\epsilon_z,\qquad
\|\nabla_y g(x^+,z_{t+1}^{K_{t+1}})\|\le\epsilon_z.
\]
Since $\epsilon_z=T^{-2}$, Lemma \ref{lem:bt:value} gives
\[
|V_t-\varphi(x_t)|\le\frac{l_f}{\mu T^2},\qquad
|V^+-\varphi(x^+)|\le\frac{l_f}{\mu T^2}.
\]
Consequently, we have
\begin{align*}
V^+-V_t
&=(V^+-\varphi(x^+))+(\varphi(x^+)-\varphi(x_t))
+(\varphi(x_t)-V_t)\\
&\le\varphi(x^+)-\varphi(x_t)+\frac{2l_f}{\mu T^2}\\
&\le-\frac{\|\widehat{\nabla}\varphi(x_t,y_t^{N_t},z_t^{K_t})\|^2}{4a}
+\frac{\hat C}{aT^2}+\frac{2l_f}{\mu T^2},
\end{align*}
where the last inequality follows from \eqref{equ:bt:descent}.

By Remark \ref{rem:epsilon}, we have $T \ge 4l_f / \mu$. Hence, if also $a \ge 4\hat C / T$, then
\[
\frac{\hat C}{aT^2}+\frac{2l_f}{\mu T^2}
\le\frac{1}{4T}+\frac{1}{2T}<\frac{1}{T}.
\]
Substituting this bound into the preceding inequality yields
\begin{equation}
\label{equ:V+V-t}
V^+\le V_t-\frac{\|\widehat{\nabla}\varphi(x_t,y_t^{N_t},z_t^{K_t})\|^2}{4a}
+\frac{1}{T}.
\end{equation}
Thus, the rejection condition in Line \ref{alg:bt:test} of Algorithm \ref{alg:AF2BA} is false whenever
\[
a\ge\max\left\{L_\varphi,\frac{4\hat C}{T}\right\}.
\]
Since $a$ is doubled whenever that condition is true, the while loop terminates after finitely many doublings.

We next bound the accepted values of $a$. The rejection condition in Line \ref{alg:bt:test} is false whenever $a\ge\max\{L_\varphi,4\hat C/T\}$, so a rejected trial must have $a<\max\{L_\varphi,4\hat C/T\}$. By Line \ref{alg:bt:double}, the next trial then uses $2a<2\max\{L_\varphi,4\hat C/T\}$. Therefore, it holds that
\begin{equation}
\label{equ:bt:abound}
1\le a_t\le2\max\left\{1,L_\varphi,\frac{4\hat C}{T}\right\}.
\end{equation}
Let $R_t$ be the number of rejected trials in outer iteration $t$, and let $R=\sum_{t=0}^{T-1}R_t$. The scalar $a$ is initialized to $1$ and is not reset between outer iterations. Since each rejection doubles $a$, the value of $a_t$ at the end of iteration $t$ is $a_t=2^{\sum_{s=0}^{t}R_s}$. Therefore, the value of $a$ at the last iteration $T - 1$ satisfies
\[
2^R=a_{T-1}\le2\max\left\{1,L_\varphi,\frac{4\hat C}{T}\right\}.
\]
Then, we have
\begin{equation}
\label{equ:Rmax}
\begin{aligned}
R\le\left\lceil\log_2\left(2\max\left\{1,L_\varphi,\frac{4\hat C}{T}\right\}\right)\right\rceil \le\left\lceil\log_2\left(2\max\left\{1,L_\varphi,4\hat C\right\}\right)\right\rceil,
\end{aligned}
\end{equation}
where the last inequality follows from $T\ge1$.

At each outer iteration, the while condition in Line \ref{alg:bt:test} of Algorithm \ref{alg:AF2BA} is false at the accepted point. Substituting $V^+=V_{t+1}$ and $a=a_t$, we have
\[
V_{t+1}\le V_t-\frac{\|\widehat{\nabla}\varphi(x_t,y_t^{N_t},z_t^{K_t})\|^2}{4a_t}
+\frac{1}{T}.
\]
Rearranging and summing over $t=0,\ldots,T-1$, we obtain
\begin{equation}
\label{equ:bt:telescoping}
\begin{aligned}
\sum_{t=0}^{T-1}\frac{\|\widehat{\nabla}\varphi(x_t,y_t^{N_t},z_t^{K_t})\|^2}{4a_t}
&\le\sum_{t=0}^{T-1}(V_t-V_{t+1})+\sum_{t=0}^{T-1}\frac{1}{T}\\
&=V_0-V_T+1\\
&\le\varphi(x_0)-\varphi^*+1+\frac{2l_f}{\mu T^2}\\
&\le\frac{3}{2}\bigl(\varphi(x_0)-\varphi^*+1\bigr),
\end{aligned}
\end{equation}
where the second inequality follows from Lemma \ref{lem:bt:value} and $\varphi(x_T)\ge\varphi^*$, and the last follows from $\varphi(x_0)\ge\varphi^*$ and the horizon condition \eqref{equ:bt:horizon}, which gives $\frac{2l_f}{\mu T^2}\le\frac{1}{2T}\le\frac{1}{2}$. Combining \eqref{equ:bt:abound} and \eqref{equ:bt:telescoping} gives
\[
\frac{1}{T}\sum_{t=0}^{T-1}\|\widehat{\nabla}\varphi(x_t,y_t^{N_t},z_t^{K_t})\|^2
\le\frac{12}{T}\max\left\{1,L_\varphi,\frac{4\hat C}{T}\right\}
\bigl(\varphi(x_0)-\varphi^*+1\bigr).
\]
For every $t$, \eqref{equ:hatC} also gives
\[
\|\nabla\varphi(x_t)\|^2
\le2\|\widehat{\nabla}\varphi(x_t,y_t^{N_t},z_t^{K_t})\|^2+\frac{2\hat C}{T^2}.
\]
Taking the average of this inequality and using the preceding bound on the average squared approximate hypergradient, we obtain
\begin{align*}
\frac{1}{T}\sum_{t=0}^{T-1}\|\nabla\varphi(x_t)\|^2
&\le\frac{2}{T}\sum_{t=0}^{T-1}\|\widehat{\nabla}\varphi(x_t,y_t^{N_t},z_t^{K_t})\|^2+\frac{2\hat C}{T^2}\\
&\le\frac{24}{T}\max\left\{1,L_\varphi,\frac{4\hat C}{T}\right\}
\bigl(\varphi(x_0)-\varphi^*+1\bigr)+\frac{2\hat C}{T^2}\\
&\le\frac{24}{T}\left(1+L_\varphi+\frac{4\hat C}{T}\right)
\bigl(\varphi(x_0)-\varphi^*+1\bigr)+\frac{2\hat C}{T^2}\\
&\le\left(\frac{24(1+L_\varphi)}{T}+\frac{98\hat C}{T^2}\right)
\bigl(\varphi(x_0)-\varphi^*+1\bigr),
\end{align*}
where the third inequality uses $\max\{1,L_\varphi,4\hat C/T\}\le1+L_\varphi+4\hat C/T$, and the last uses $\varphi(x_0)-\varphi^*+1\ge1$. The proof of \eqref{equ:bt:mainbound} is complete.

Algorithm \ref{alg:AF2BA} calls \hyperref[alg:AC-GM]{AC-GM} $T$ times for the penalty problem and $T+R+1$ times for the lower-level problem, including the initialization and rejected trials, where $R$ is the total number of rejected trials and satisfies \eqref{equ:Rmax}.

By Proposition \ref{prop:KtNt2}, the total number of first-order oracle queries, including residual checks and the remaining function and gradient evaluations, is bounded by
\[
\begin{aligned}
c\Bigl[T(1+\max N_t)+(T+R+1)(1+\max K_t)\Bigr],
\end{aligned}
\]
where $c>0$ is a constant and the maxima are taken over all calls for the corresponding problems, including initialization and rejected trials.
Proposition \ref{prop:KtNt2} gives $K_t,N_t=\gO(1+\log T)$, so the total query bound is $\gO(T(1+\log T))$. The desired result follows.
\end{proof}

\subsection{{Proof of Corollary \ref{cor:condition}}}
\begin{proof}
By \cite[Lemma 4.3]{chen2024finding}, the penalty approximation constant in Lemma \ref{lem:nablalimit} satisfies
\[
\bar D_\sigma=\gO(\ell\kappa^3).
\]
Lemma \ref{lem:nablavarphi} gives $L_\varphi\le\ell(1+\kappa)^3=\gO(\ell\kappa^3)$. By the definitions of $\hat C$ and $C_\varphi$ in \eqref{equ:hatC} and \eqref{equ:Cphi}, we have
\[
\hat C=\gO(\ell^2\kappa^6),\qquad
C_\varphi=\gO(\ell\kappa^3).
\]
Since $\varphi(x_0)-\varphi^*=\gO(1)$, Theorem \ref{thm:conv} yields, for some constant $C>0$,
\[
\frac1T\sum_{t=0}^{T-1}\|\nabla\varphi(x_t)\|^2
\le C\left(\frac{\ell\kappa^3}{T}+\frac{\ell^2\kappa^6}{T^2}\right).
\]
For $0<\epsilon\le1$, taking $T=\lceil \ell\kappa^3\epsilon^{-2}\rceil$ yields
\[
\frac{\ell\kappa^3}{T}\le \epsilon^2,
\qquad
\frac{\ell^2\kappa^6}{T^2}\le \epsilon^4
\le \epsilon^2.
\]
Thus, we have
\[
\frac1T\sum_{t=0}^{T-1}\|\nabla\varphi(x_t)\|^2
\le C\epsilon^2.
\]
Therefore, the iteration complexity is $T = \gO(\ell\kappa^3\epsilon^{-2})$.

We next bound the total number of first-order oracle queries. The definitions in the proof of Proposition \ref{prop:KtNt2} give $p_h^{-1}=\gO(\kappa)$, $m_h=\gO(1)$, and $\bar C_h=\gO(\kappa)$. Hence $C_h$ is polynomially bounded in $\kappa$. Since $C_\varphi=\gO(\ell\kappa^3)$, the definitions of $\bar C_b$ and $\bar C_c$ yield
\[
\log\bar C_b,\log\bar C_c=\gO\bigl(\log(2\ell\kappa)\bigr).
\]
Substituting these results and $\epsilon_y=\epsilon_z=T^{-2}$ into the iteration bounds in Proposition \ref{prop:KtNt2}, and using $\log(1/(1-p_h))\ge p_h$, we have
\[
K_t,N_t=\gO \left(\kappa\left[\log(2\ell\kappa)+4\log T\right]\right) = \gO (\kappa \log T).
\]
The proof of Theorem \ref{thm:conv} gives $R=\gO(\log(2\ell\kappa))$. Then, the total number of first-order oracle queries is
\[
\gO\Bigl(T(1+\max N_t)+(T+R+1)(1+\max K_t)\Bigr) = \widetilde{\gO}(\ell\kappa^4\epsilon^{-2}).
\]
The proof is complete.
\end{proof}

\section{{Numerical Experiments}}
In this section, we compare our AF$^2$BA with several baselines: V-PBGD~\citep{shen2023penalty}, BOME~\citep{liu2022bome}, F$^2$BA~\citep{chen2024finding}, Prox-F$^2$BA~\citep{kwon2023penalty}, and GALET~\citep{xiao2023generalized} using the synthetic problem in Example~\ref{ex:synthetic} and two machine learning tasks: data hyper-cleaning and feature-wise $\ell_2$ hyperparameter optimization. The source and configuration of each task are specified below.

For AF$^2$BA, we set $\sigma=10^{-3}$ and use $\epsilon_y=\epsilon_z=10^{-6}$ as the residual tolerances for the penalty and lower-level subproblems, respectively. For each experiment, we first compare the methods at a fixed initial stepsize $s$ (i.e., $a=1/s$ in Algorithm \ref{alg:AF2BA}) and then vary $s$ to assess sensitivity. Here, for each baseline method the same initial stepsize $s$ is used in its outer and inner gradient steps.

\subsection{Synthetic example}
In this experiment, we consider the synthetic problem in Example~\ref{ex:synthetic}. Here, $x\in\mathbb R^2$ and $y\in\mathbb R^3$; the third component of \(y\) does not affect the lower-level objective and is therefore unrestricted at a lower-level minimizer, making the lower-level solution set non-singleton. Since Example~\ref{ex:synthetic} gives an explicit expression for \(\varphi\), we evaluate \(\varphi(x_k)\) directly and compute its gradient analytically to assess stationarity. We run each method for up to 2,000 outer iterations and test sensitivity to the following initial stepsizes:
\[
 s\in\{0.05,0.1,1,5,10,100\}.
\]
Figure~\ref{fig:uniformpl-fixed} compares the explicit value $\varphi(x_k)$ and the exact hypergradient norm at $s=0.05$. AF$^2$BA attains a hyperobjective value comparable to that of F\(^2\)BA, and its hypergradient norm decreases over the iterations. Figure~\ref{fig:uniformpl-sweep} examines the effect of different initial stepsizes. AF$^2$BA remains stable for every tested initial stepsize and attains similar hyperobjective values, whereas other methods exhibit oscillations or divergence at larger stepsizes. These results show that AF$^2$BA achieves comparable performance and remains stable across the tested initial stepsizes.

\begin{figure}[htp]
\centering
\begin{minipage}[t]{0.4\linewidth}
\centering
\includegraphics[width=1\linewidth]{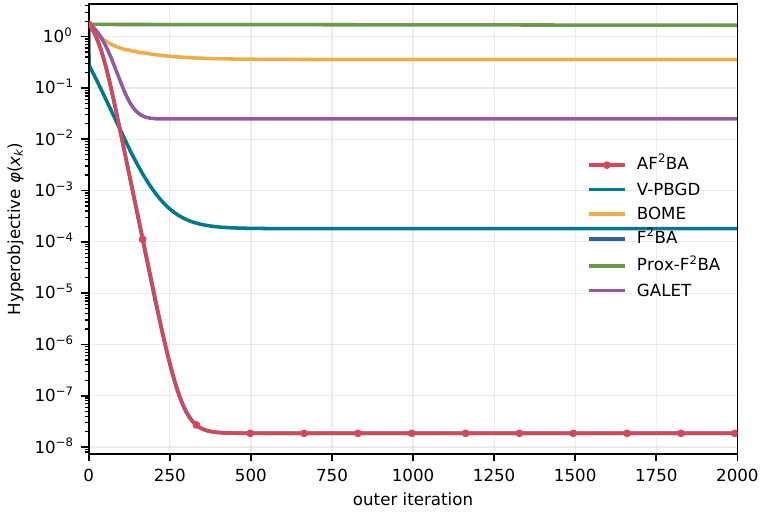}
\par\vspace{2pt}{\small (a)}
\end{minipage}\hfill
\begin{minipage}[t]{0.4\linewidth}
\centering
\includegraphics[width=1\linewidth]{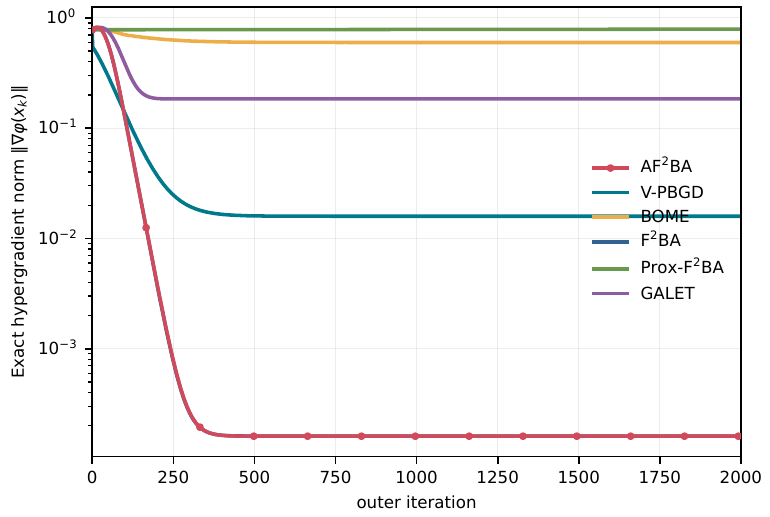}
\par\vspace{2pt}{\small (b)}
\end{minipage}
\caption{Synthetic example at $s=0.05$: (a) explicit hyperobjective $\varphi(x_k)$; (b) exact hypergradient norm.}
\label{fig:uniformpl-fixed}
\end{figure}

\begin{figure}[htp]
\centering
{\scriptsize $s$: \textcolor[rgb]{0.267968,0.223549,0.512008}{\rule[0.4ex]{9pt}{0.65pt}}\,$0.05$ \textcolor[rgb]{0.190631,0.407061,0.556089}{\rule[0.4ex]{9pt}{0.65pt}}\,$0.1$ \textcolor[rgb]{0.127568,0.566949,0.550556}{\rule[0.4ex]{9pt}{0.65pt}}\,$1$ \textcolor[rgb]{0.208030,0.718701,0.472873}{\rule[0.4ex]{9pt}{0.65pt}}\,$5$ \textcolor[rgb]{0.565498,0.842430,0.262877}{\rule[0.4ex]{9pt}{0.65pt}}\,$10$ \textcolor[rgb]{0.993248,0.906157,0.143936}{\rule[0.4ex]{9pt}{0.65pt}}\,$100$}\par\vspace{5pt}
\includegraphics[width=.32\linewidth]{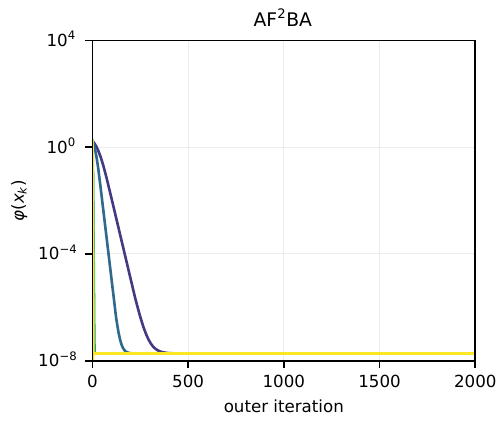}%
\hfill%
\includegraphics[width=.32\linewidth]{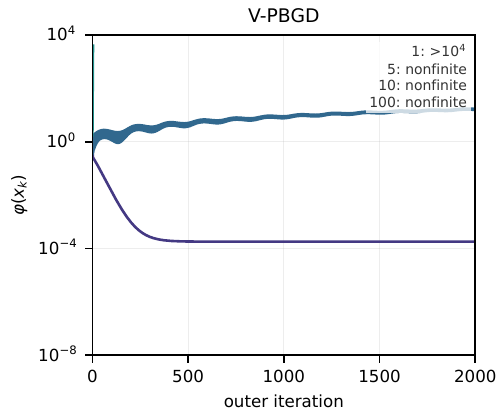}%
\hfill%
\includegraphics[width=.32\linewidth]{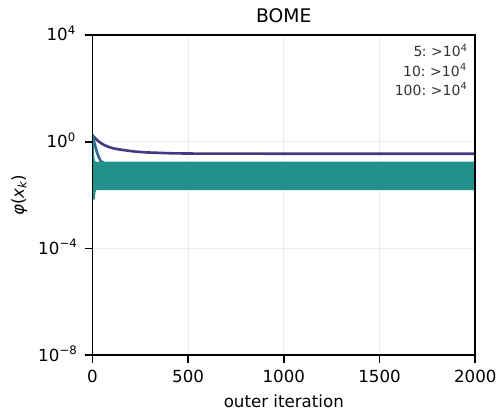}\par\vspace{5pt}
\includegraphics[width=.32\linewidth]{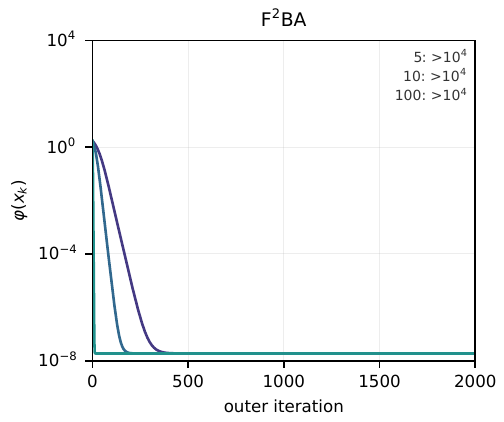}%
\hfill%
\includegraphics[width=.32\linewidth]{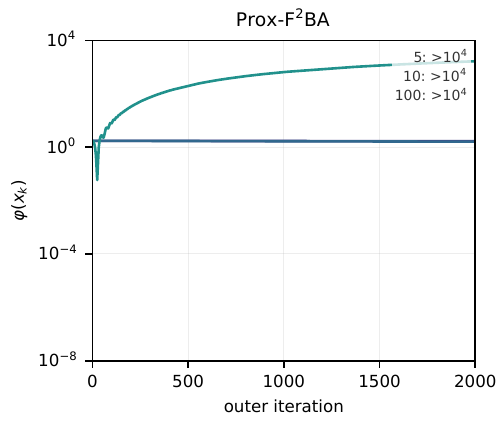}%
\hfill%
\includegraphics[width=.32\linewidth]{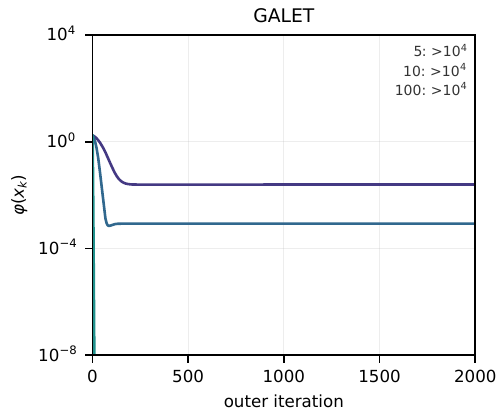}
\caption{Explicit hyperobjective $\varphi(x_k)$ for the uniform P{\L} synthetic example at different initial stepsizes. Annotations mark trajectories that exceed the displayed range $[10^{-8},10^4]$; nonfinite runs are identified separately.}
\label{fig:uniformpl-sweep}
\end{figure}

\subsection{Data hyper-cleaning}
\label{sec:hyperclean}
Data hyper-cleaning learns training-sample weights from clean validation data \citep{franceschi2018bilevel,shaban2019truncated,liu2022bome,shen2023penalty}. We consider a linear-model MNIST data hyper-cleaning problem based on the setup in \cite{shen2023penalty}. We use smaller data subsets and include an $\ell_2$-regularization term in the lower-level objective. We draw disjoint subsets of 2000 training, 2000 validation, and 4000 test examples from the MNIST training pool. Half of the training labels are replaced by a different class; validation and test labels remain clean. With sample weights $w_i(x)=\operatorname{sigmoid}(x_i)$, the bilevel problem is formulated as
\[
\begin{aligned}
\min_{x\in\mathbb R^{2000}}\quad
&\frac1{2000}\sum_{i\in\mathrm{val}}\mathcal L_{\mathrm{L},i}^{\rm clean}(\Theta^*(x)P),\\
\text{s.t.}\quad
&\Theta^*(x)\in\arg\min_{\Theta\in\mathbb R^{785\times10}}
\left\{\frac1{2000}\sum_{i\in\mathrm{tr}}w_i(x)\mathcal L_{\mathrm{L},i}^{\rm noisy}(\Theta P)
+\frac{0.01}{2}\|\Theta P\|_F^2\right\}.
\end{aligned}
\]
Here $\mathcal L_{\mathrm{L},i}$ denotes the cross-entropy loss for sample $i$ in the linear classifier, and $P=I-\frac1{10}\bone\bone^\top$ centers the class coefficients. The 785 features comprise 784 pixels and a bias coordinate; pixels are divided by 255 and centered using the training-feature mean. The ridge term makes the lower objective strongly convex on the centered subspace. Since the validation loss is convex in the classifier, both $g$ and $g+\sigma f$ are P{\L} in the full parameter space, with flat class-shift directions. We test the initial stepsizes \(s\in\{0.01,0.05,0.1,1,5,10,100\}\) and run each method for 1,000 outer iterations. We report the validation loss and the F1 score, which evaluates the identification of clean training examples using the learned weights.

Figure~\ref{fig:hyperclean-fixed} compares the validation loss and F1 score at \(s=0.1\). At this stepsize, V-PBGD attains a lower validation loss than the other methods, while AF\(^2\)BA achieves results comparable to those of the other baselines. Figure~\ref{fig:hyperclean-sweep} shows that AF\(^2\)BA remains stable across the tested initial stepsizes, whereas several baselines diverge at \(s=10\) or \(s=100\), which demonstrates the robustness of our method.

\begin{figure}[htp]
\centering
\begin{minipage}[t]{0.4\linewidth}
\centering
\includegraphics[width=1\linewidth]{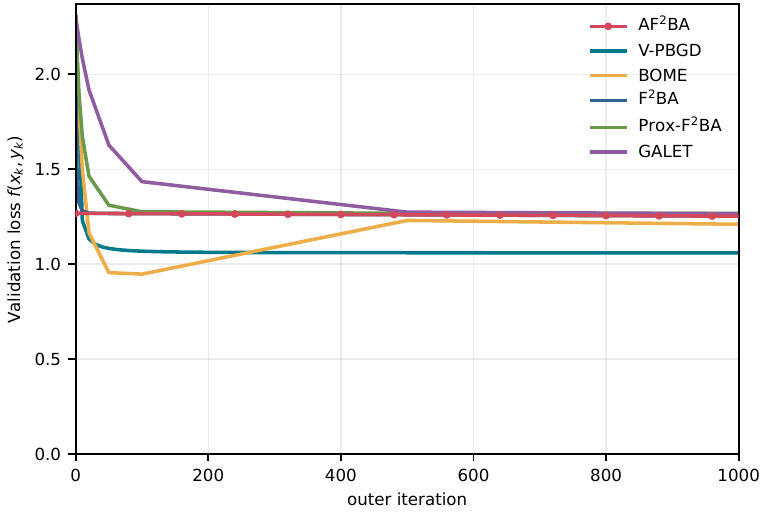}
\par\vspace{2pt}{\small (a)}
\end{minipage}\hfill
\begin{minipage}[t]{0.4\linewidth}
\centering
\includegraphics[width=1\linewidth]{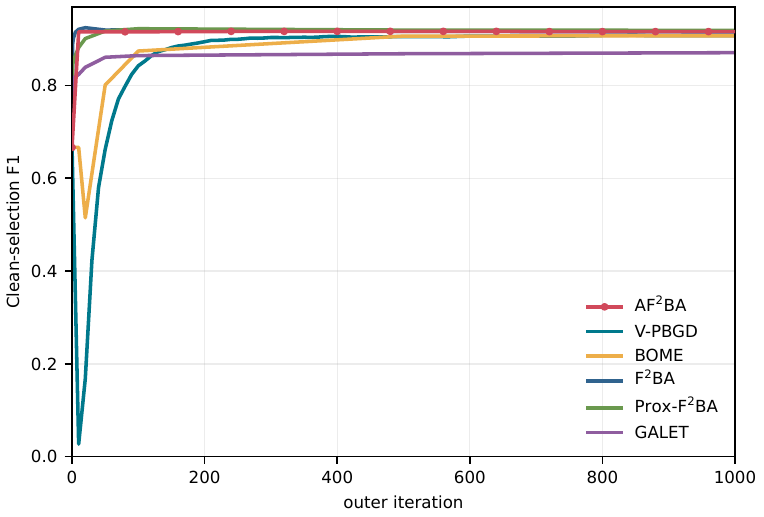}
\par\vspace{2pt}{\small (b)}
\end{minipage}
\caption{Data hyper-cleaning at $s=0.1$: (a) validation loss $f(x_k,y_k)$; (b) F1 score.}
\label{fig:hyperclean-fixed}
\end{figure}

\begin{figure}[htp]
\centering
{\scriptsize $s$: \textcolor[rgb]{0.267004,0.004874,0.329415}{\rule[0.4ex]{9pt}{0.65pt}}\,$0.01$ \textcolor[rgb]{0.267968,0.223549,0.512008}{\rule[0.4ex]{9pt}{0.65pt}}\,$0.05$ \textcolor[rgb]{0.190631,0.407061,0.556089}{\rule[0.4ex]{9pt}{0.65pt}}\,$0.1$ \textcolor[rgb]{0.127568,0.566949,0.550556}{\rule[0.4ex]{9pt}{0.65pt}}\,$1$ \textcolor[rgb]{0.208030,0.718701,0.472873}{\rule[0.4ex]{9pt}{0.65pt}}\,$5$ \textcolor[rgb]{0.565498,0.842430,0.262877}{\rule[0.4ex]{9pt}{0.65pt}}\,$10$ \textcolor[rgb]{0.993248,0.906157,0.143936}{\rule[0.4ex]{9pt}{0.65pt}}\,$100$}\par\vspace{5pt}
\includegraphics[width=.32\linewidth]{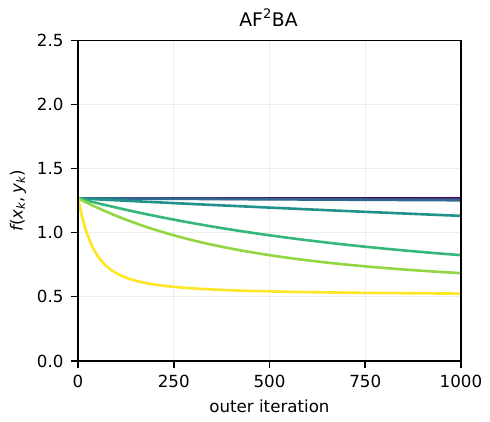}%
\hfill%
\includegraphics[width=.32\linewidth]{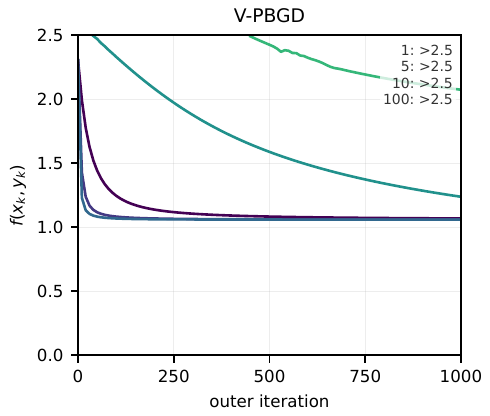}%
\hfill%
\includegraphics[width=.32\linewidth]{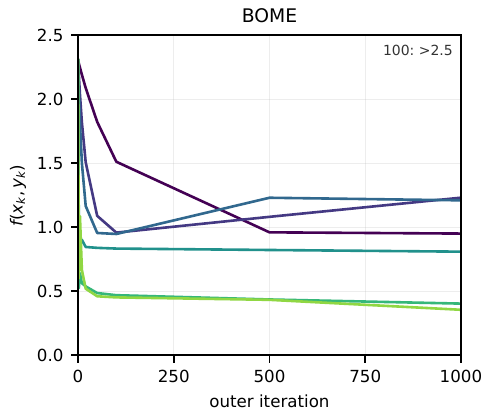}\par\vspace{5pt}
\includegraphics[width=.32\linewidth]{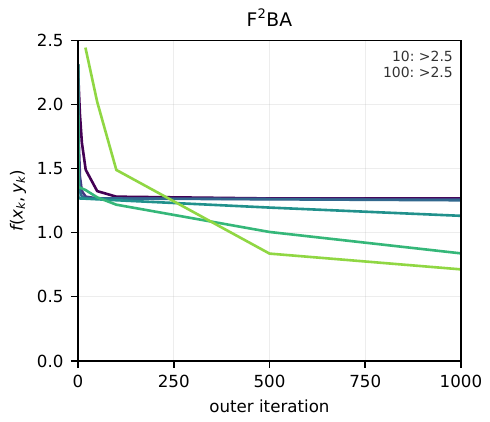}%
\hfill%
\includegraphics[width=.32\linewidth]{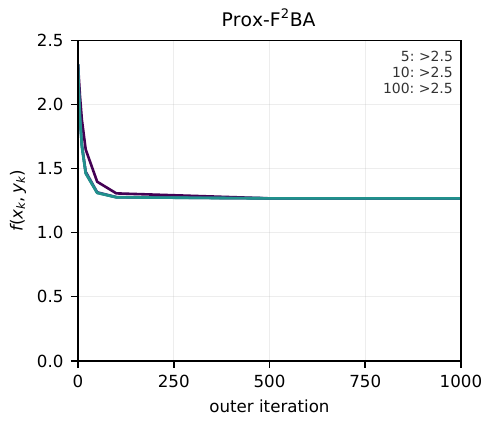}%
\hfill%
\includegraphics[width=.32\linewidth]{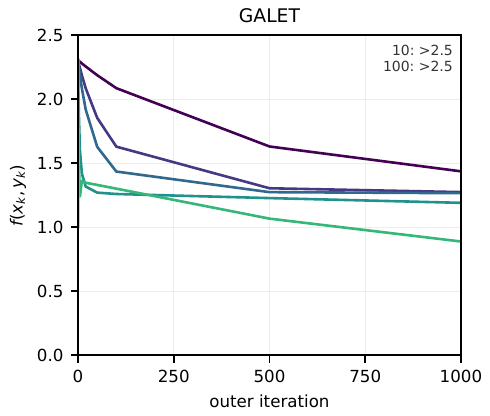}
\caption{Validation loss $f(x_k,y_k)$ for data hyper-cleaning at different initial stepsizes. Annotations mark trajectories that exceed the displayed range \([0,2.5]\).}
\label{fig:hyperclean-sweep}
\end{figure}

\subsection{Data hyper-cleaning with a multilayer perceptron}
\label{sec:hyperclean-mlp}
Following the MLP experiment in Section~5.2 of \citep{shen2023penalty}, we replace the linear classifier in Section~\ref{sec:hyperclean} with a network having 784 inputs, a sigmoid hidden layer of width 300, and 10 output logits. The data split, label corruption, preprocessing, and sample-weight parameterization are the same as in Section~\ref{sec:hyperclean}. All methods start from the same network parameters and sample weights, with the network initialized using a fixed random seed.

Let $N_\theta(u)=W_2\operatorname{sigmoid}(W_1u+b_1)+b_2$, where $\theta$ collects the 238,510 network parameters. With $\mathcal L_{\mathrm{CE}}$ denoting the cross-entropy loss, $c_i$ the clean validation label, and $\widetilde c_i$ the noisy training label, the bilevel problem is
\[
\begin{aligned}
\min_{x\in\mathbb R^{2000}}\quad
& f(x,\theta^*(x)):=\frac1{2000}\sum_{i\in\mathrm{val}}
\mathcal L_{\mathrm{CE}}(N_{\theta^*(x)}(u_i),c_i),\\
\text{s.t.}\quad
&\theta^*(x)\in\arg\min_{\theta\in\mathbb R^{238510}}
\left\{\frac1{2000}\sum_{i\in\mathrm{tr}}w_i(x)
\mathcal L_{\mathrm{CE}}(N_\theta(u_i),\widetilde c_i)+\frac{0.01}{2}\|\theta\|^2\right\}.
\end{aligned}
\]
The ridge term is applied to all network parameters, including the biases. Unlike the linear classifier, the MLP generally gives nonconvex objectives in $\theta$; the ridge term alone does not establish a P{\L} condition for $g$ or a uniform P{\L} condition for $g+\sigma f$. We therefore include this experiment to examine the numerical behavior of AF$^2$BA beyond the setting in which the uniform P{\L} assumption has been verified, rather than as a test of the convergence theorem.

Consistent with the preceding experiments, Figure~\ref{fig:hyperclean-mlp-fixed} shows that AF$^2$BA achieves performance comparable to that of the baselines. Figure~\ref{fig:hyperclean-mlp-sweep} further indicates stable behavior of AF$^2$BA across all tested initial stepsizes, while several baselines become unstable at larger initial stepsizes, demonstrating its robustness. These results also indicate the practical effectiveness of AF$^2$BA on this MLP task, for which the uniform P{\L} condition has not been established.

\begin{figure}[htp]
\centering
\begin{minipage}[t]{0.4\linewidth}
\centering
\includegraphics[width=1\linewidth]{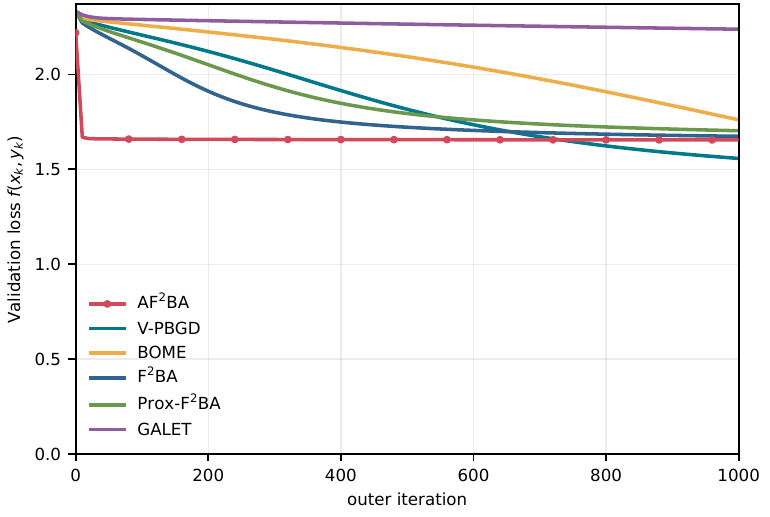}
\par\vspace{2pt}{\small (a)}
\end{minipage}\hfill
\begin{minipage}[t]{0.4\linewidth}
\centering
\includegraphics[width=1\linewidth]{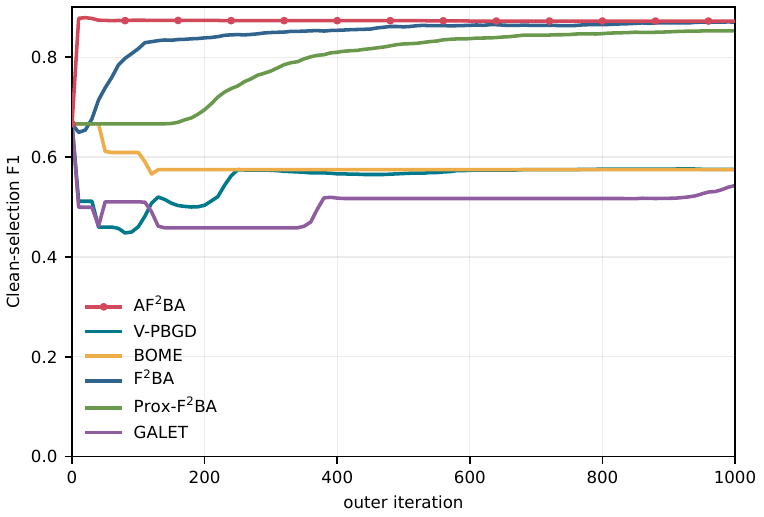}
\par\vspace{2pt}{\small (b)}
\end{minipage}
\caption{MLP data hyper-cleaning at $s=0.01$: (a) validation loss $f(x_k,y_k)$; (b) F1 score for identifying clean training examples.}
\label{fig:hyperclean-mlp-fixed}
\end{figure}

\begin{figure}[htp]
\centering
{\scriptsize $s$: \textcolor[rgb]{0.267004,0.004874,0.329415}{\rule[0.4ex]{9pt}{0.65pt}}\,$0.01$ \textcolor[rgb]{0.267968,0.223549,0.512008}{\rule[0.4ex]{9pt}{0.65pt}}\,$0.05$ \textcolor[rgb]{0.190631,0.407061,0.556089}{\rule[0.4ex]{9pt}{0.65pt}}\,$0.1$ \textcolor[rgb]{0.127568,0.566949,0.550556}{\rule[0.4ex]{9pt}{0.65pt}}\,$1$ \textcolor[rgb]{0.208030,0.718701,0.472873}{\rule[0.4ex]{9pt}{0.65pt}}\,$5$ \textcolor[rgb]{0.565498,0.842430,0.262877}{\rule[0.4ex]{9pt}{0.65pt}}\,$10$ \textcolor[rgb]{0.993248,0.906157,0.143936}{\rule[0.4ex]{9pt}{0.65pt}}\,$100$}\par\vspace{5pt}
\includegraphics[width=.32\linewidth]{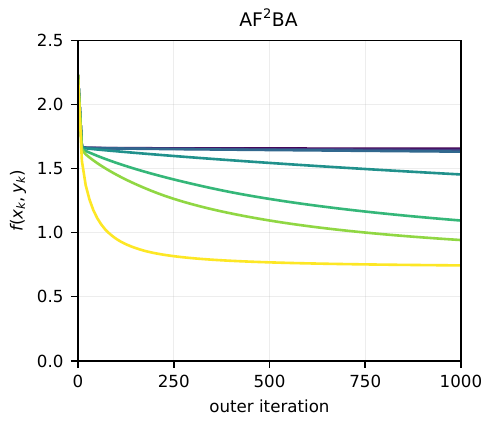}%
\hfill%
\includegraphics[width=.32\linewidth]{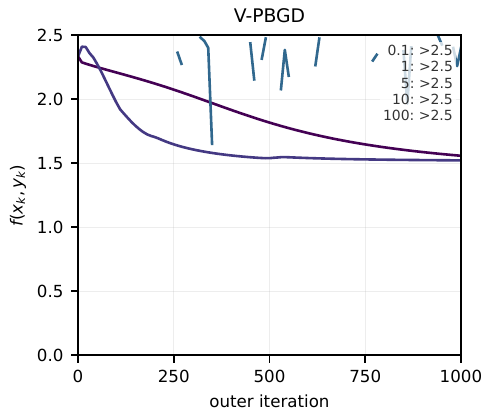}%
\hfill%
\includegraphics[width=.32\linewidth]{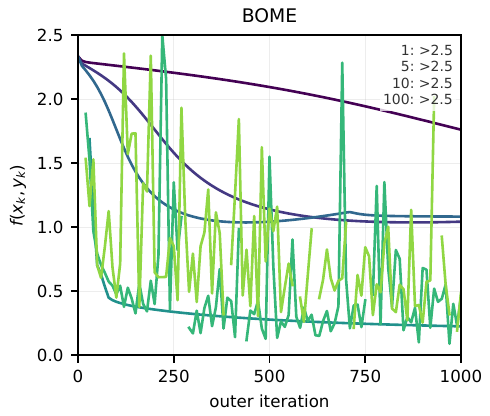}\par\vspace{5pt}
\includegraphics[width=.32\linewidth]{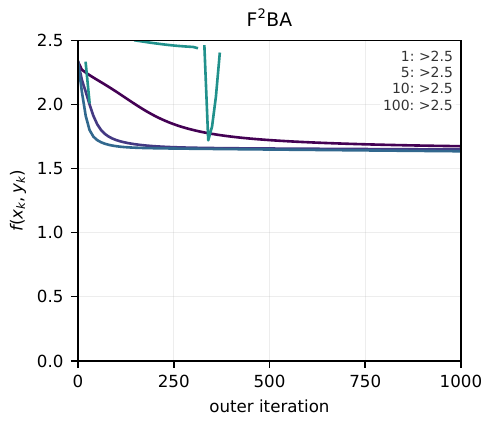}%
\hfill%
\includegraphics[width=.32\linewidth]{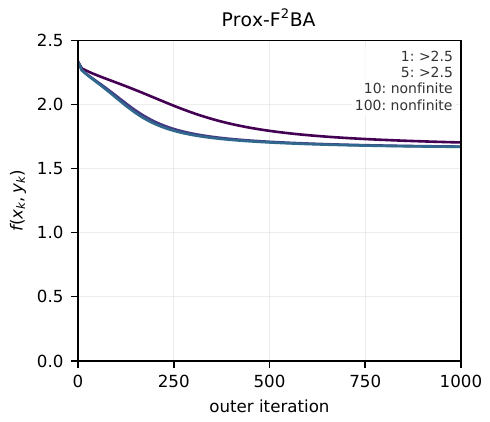}%
\hfill%
\includegraphics[width=.32\linewidth]{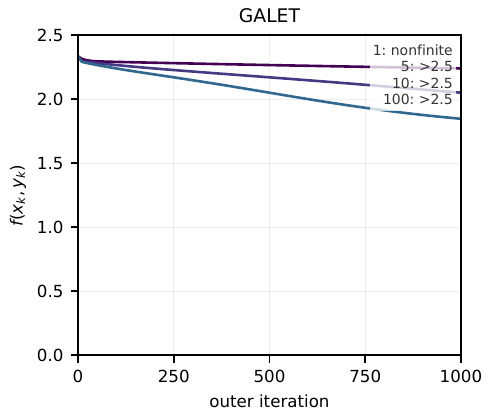}
\caption{Validation loss $f(x_k,y_k)$ for MLP data hyper-cleaning at different initial stepsizes. All axes are linear. Annotations identify trajectories above the displayed range $[0,2.5]$ and nonfinite updates; values outside this range are omitted only from the display, and finite runs continue to the full budget.}
\label{fig:hyperclean-mlp-sweep}
\end{figure}

\subsection{Feature-wise $\ell_2$ hyperparameter optimization}
\label{sec:hpo}
We consider a feature-wise regularization problem following the setup in \cite{pedregosa2016hyperparameter}, with a centered softmax classifier and feature-wise regularization coefficients bounded below by $0.001$. We use the 1,797-sample UCI Optical Digits dataset distributed with scikit-learn\footnote{UCI Optical Recognition of Handwritten Digits dataset: \url{https://doi.org/10.24432/C50P49}.}. Our experimental configuration uses a 200/150/300 train/validation/test split and the bounded parameterization specified below. Each feature has one regularization coefficient shared across classes, $\rho_j(x)=0.001+0.2\operatorname{sigmoid}(x_j)$. The bilevel problem has the following form:
\[
\begin{aligned}
\min_{x\in\mathbb R^{65}}\quad
&\frac1{150}\sum_{i\in\mathrm{val}}\mathcal L_{\mathrm{L},i}^{\rm clean}(\Theta^*(x)P),\\
\text{s.t.}\quad
&\Theta^*(x)\in\arg\min_{\Theta\in\mathbb R^{65\times10}}
\left\{\frac1{200}\sum_{i\in\mathrm{tr}}\mathcal L_{\mathrm{L},i}^{\rm train}(\Theta P)
+\frac12\sum_{j=1}^{65}\rho_j(x)\|(\Theta P)_{j,:}\|_2^2\right\}.
\end{aligned}
\]
The 65 features comprise 64 pixel values and a bias coordinate. Pixels are divided by 16 and centered using the training-feature mean. We use the same class-centering matrix $P$. Since \(\rho_j(x)\geq 0.001\), the lower-level objective is strongly convex on the centered classifier subspace. Thus, \(g\) and \(g+\sigma f\) satisfy a P{\L} condition on the full parameter space, while remaining invariant along class-shift directions. We use the same number of outer iterations (1,000) and initial stepsizes as in the second experiment. Here, we report the validation loss and test macro-F1 score, the latter being the average of the class-wise F1 scores on the test set.

Figure~\ref{fig:hpo-fixed} shows that the validation loss and F1 score of AF$^2$BA are comparable to those of other methods. As in the preceding two experiments, Figure~\ref{fig:hpo-sweep} also shows the robustness of AF$^2$BA.

\begin{figure}[htp]
\centering
\begin{minipage}[t]{0.4\linewidth}
\centering
\includegraphics[width=1\linewidth]{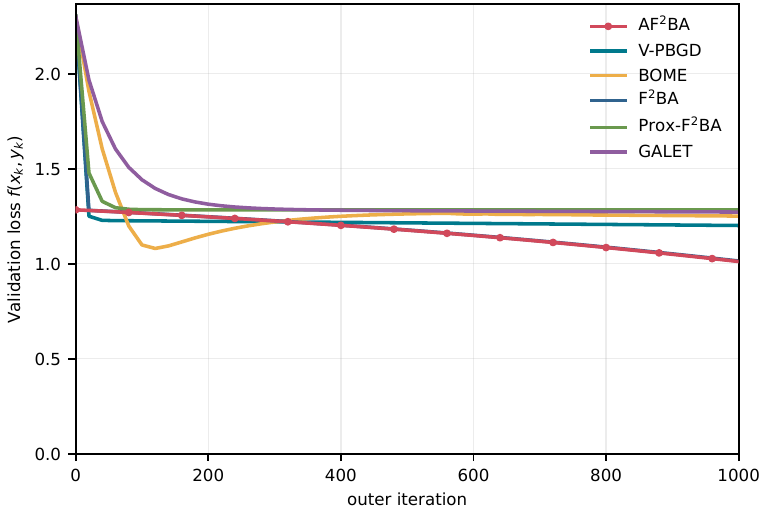}
\par\vspace{2pt}{\small (a)}
\end{minipage}\hfill
\begin{minipage}[t]{0.4\linewidth}
\centering
\includegraphics[width=1\linewidth]{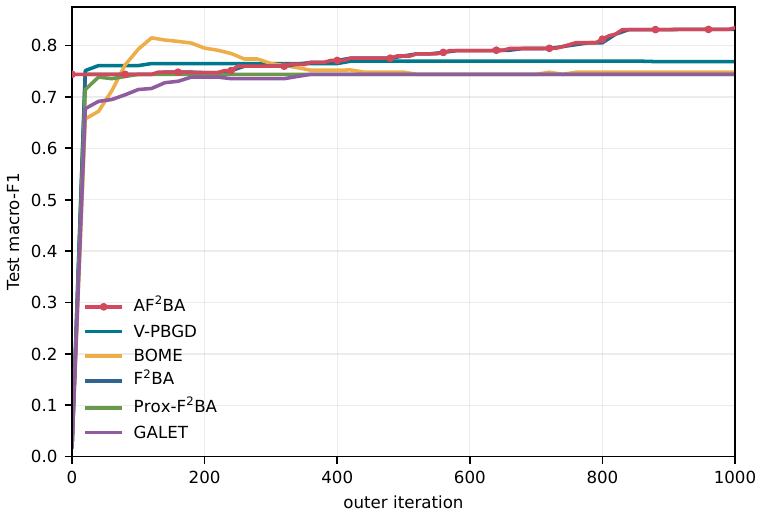}
\par\vspace{2pt}{\small (b)}
\end{minipage}
\caption{Feature-wise $\ell_2$ hyperparameter optimization at $s=0.1$: (a) validation loss $f(x_k,y_k)$; (b) F1 score.}
\label{fig:hpo-fixed}
\end{figure}

\begin{figure}[htp]
\centering
{\scriptsize $s$: \textcolor[rgb]{0.267004,0.004874,0.329415}{\rule[0.4ex]{9pt}{0.65pt}}\,$0.01$ \textcolor[rgb]{0.267968,0.223549,0.512008}{\rule[0.4ex]{9pt}{0.65pt}}\,$0.05$ \textcolor[rgb]{0.190631,0.407061,0.556089}{\rule[0.4ex]{9pt}{0.65pt}}\,$0.1$ \textcolor[rgb]{0.127568,0.566949,0.550556}{\rule[0.4ex]{9pt}{0.65pt}}\,$1$ \textcolor[rgb]{0.208030,0.718701,0.472873}{\rule[0.4ex]{9pt}{0.65pt}}\,$5$ \textcolor[rgb]{0.565498,0.842430,0.262877}{\rule[0.4ex]{9pt}{0.65pt}}\,$10$ \textcolor[rgb]{0.993248,0.906157,0.143936}{\rule[0.4ex]{9pt}{0.65pt}}\,$100$}\par\vspace{5pt}
\includegraphics[width=.32\linewidth]{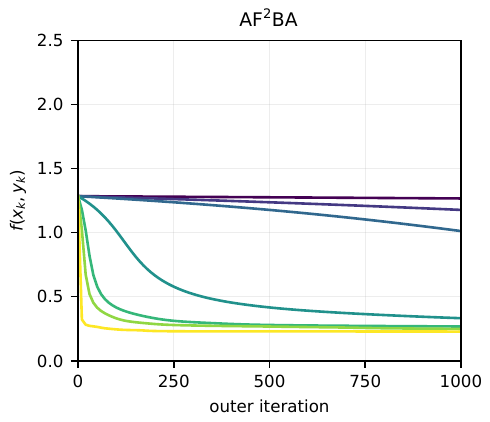}%
\hfill%
\includegraphics[width=.32\linewidth]{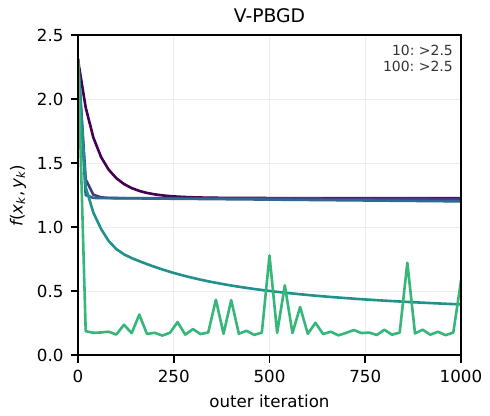}%
\hfill%
\includegraphics[width=.32\linewidth]{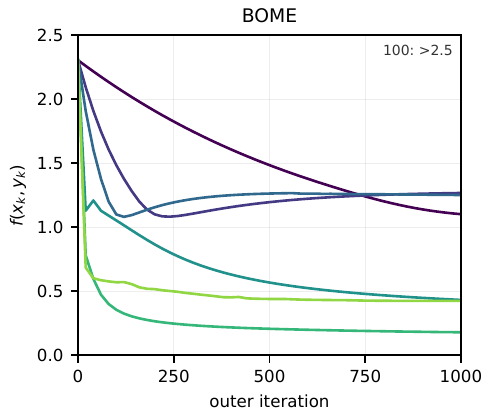}\par\vspace{5pt}
\includegraphics[width=.32\linewidth]{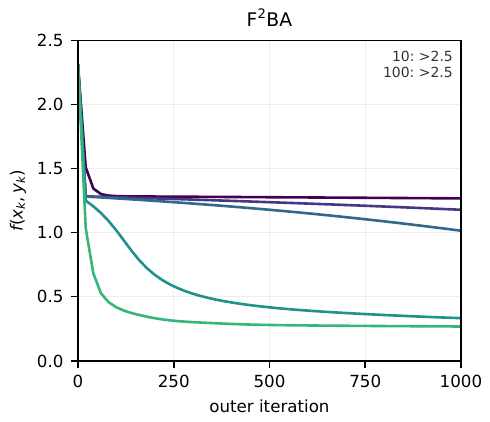}%
\hfill%
\includegraphics[width=.32\linewidth]{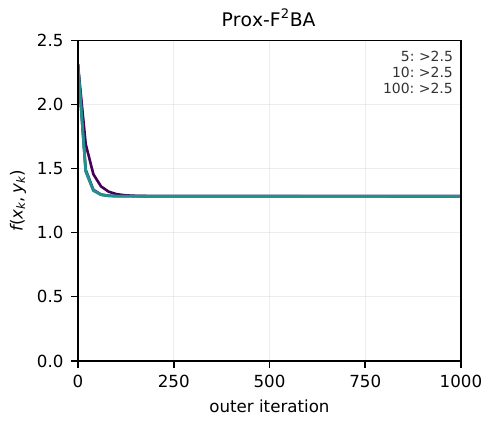}%
\hfill%
\includegraphics[width=.32\linewidth]{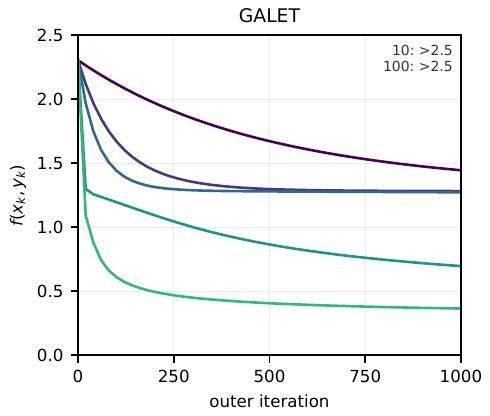}
\caption{Validation loss $f(x_k,y_k)$ for feature-wise $\ell_2$ hyperparameter optimization at different initial stepsizes. Annotations identify steps with values above $2.5$.}
\label{fig:hpo-sweep}
\end{figure}

\section{Conclusion and Future Work}
This paper proposes an adaptive first-order algorithm, termed {AF${}^2$BA, for solving nonconvex bilevel optimization (NBO) problems under the uniform Polyak--{\L}ojasiewicz (P{\L}) condition. The method chooses the outer stepsize by backtracking based on approximate hyperobjective values, and uses AC-GM to solve the inner P{\L} problems. Both stepsize rules are independent of problem-specific parameters. For a sufficiently large prescribed horizon $T$, the proposed AF${}^2$BA achieves an $\widetilde{\gO}(\ell\kappa^4\epsilon^{-2})$ first-order oracle complexity, which is near-optimal in its dependence on $\epsilon$ and matches the condition-number dependence of well-tuned F${}^2$BA. The experimental results also demonstrate the comparable performance and greater robustness of our algorithm.
}

This work introduces a double-loop adaptive algorithm for deterministic NBO problems under the {uniform P{\L} condition}. Potential extensions include: (1) extending the adaptive frameworks to stochastic settings \citep{kwon2023penalty,chen2024finding,huang2023momentum}; (2) designing single-loop adaptive algorithms \citep{kwon2023penalty,yang2025tuning}; (3) weakening or eliminating the uniform P{\L} requirement for the penalty functions \citep{shen2023penalty,xiao2023generalized,huang2023momentum}; (4) relaxing Assumption \ref{ass:basic}(i) to tackle NBO problems with more general error bound conditions, e.g., the Kurdyka--{\L}ojasiewicz conditions \citep{attouch2009convergence,attouch2010proximal,frankel2015splitting}.

\section*{Declarations}
\noindent
\textbf{Conflict of Interest:} The authors declare that they have no conflict of interest.

\noindent
\textbf{Data and Code Availability:} The experiments use the publicly available MNIST and UCI Optical Digits datasets, as described in Section 5.

\bibliographystyle{plainnat}

\bibliography{ref}

\end{document}